\documentclass[10pt,twoside, a4paper, english, reqno]{amsart}
\usepackage[dvips]{epsfig}
\usepackage{amscd}
\usepackage{a4wide}
\usepackage{amssymb}
\usepackage{amsthm}
\usepackage{amsmath}
\usepackage{bm} 
\usepackage{latexsym}

\usepackage{upref}
\usepackage{hyperref}
\usepackage{ulem}

\usepackage{enumerate}
\usepackage[usenames,dvipsnames]{xcolor}

\usepackage{subcaption}
\usepackage{mathtools}

\usepackage[foot]{amsaddr}

\theoremstyle{plain}

\newtheorem{theorem}{Theorem}[section]
\newtheorem{proposition}[theorem]{Proposition}
\newtheorem{lemma}[theorem]{Lemma}

\newtheorem{remark}[theorem]{Remark}
\newtheorem{corollary}[theorem]{Corollary}

\numberwithin{figure}{section}
\numberwithin{table}{section}

\allowdisplaybreaks

\newcounter{asnr}

{\ifnum\value{asnr}=0 \stepcounter{asnr} 
	\begin{enumerate}[label=\textbf{A}.\arabic{enumi}]
		\else
		\begin{enumerate}[label=\textbf{A}.\arabic{enumi},resume] \fi}
		{\end{enumerate}}
	
	\newcounter{defnr}

	{\ifnum\value{defnr}=0 \stepcounter{defnr} 
		\begin{enumerate}[label=\textbf{D}.\arabic{enumi}]
			\else
			\begin{enumerate}[label=\textbf{D}.\arabic{enumi},resume] \fi}
			{\end{enumerate}}

		\newcommand{\R}{\mathbb{R}}
		\newcommand{\pl}{\partial}
\newcommand{\eps}{\varepsilon}
  \newcommand{\HS}{\mathbf{H}}
\newcommand{\LS}{\mathbf{L}}
\newcommand{\CS}{\mathbf{C}}

		\numberwithin{equation}{section} \allowdisplaybreaks
		
\title[stable determination of the time-dependent matrix potential]{Stable determination of a time-dependent matrix potential for a wave equation in an infinite waveguide}
		
\author[N. Kumar, T. Sarkar and M. Vashisth]
		{Nitesh Kumar$^\dagger$, Tanmay Sarkar$^\ddagger$ and Manmohan Vashisth$^{*}$}	
		\date{}
		\address[]{$^\dagger$Department of Mathematics, 
			Indian Institute of Technology Jammu,
			NH-44 Bypass Road, Nagrota PO, Jagti, 
			Jammu - 181221, INDIA.}
		\email[]{2019rma0004@iitjammu.ac.in}
		\address[]{$^\ddagger$Department of Mathematics, 
			Indian Institute of Technology Jammu,
			NH-44 Bypass Road, Nagrota PO, Jagti, 
			Jammu - 181221, INDIA.}
		\email[]{tanmay.sarkar@iitjammu.ac.in}
		\address[]{$^{*}$Department of Mathematics,  Indian Institute of Technology Ropar, Rupnagar, Punjab-140001, INDIA.}
		\email[]{\tt\  manmohanvashisth@iitrpr.ac.in}

		
		
\begin{document}
\begin{abstract}
    We analyze the stability of an inverse problem for determining the time-dependent matrix potential appearing in the Dirichlet initial-boundary value problem for the wave equation in an unbounded cylindrical waveguide. The observation is given by the input-output map associated with the wave equation. Considering a suitable geometric optics solution and with the help of \emph{light ray transform}, we demonstrate the stability estimate in the determination of the time-dependent matrix potential from the given input-output map.
\end{abstract}
\maketitle
\section{Introduction}
The present paper is concerned with an inverse problem of determining a matrix-valued potential $\displaystyle q(t,x):=((q_{ij}(t,x)))_{1\leq i,j\leq n}$ in an unbounded cylindrical domain $\Omega = \omega \times \mathbb{R}$,
where $\omega$ is a $\mathcal{C}^\infty$ bounded open connected domain of the Euclidean space $\mathbb{R}^2$. For time $T>0$, we denote $\Omega_T:=(0,T)\times\Omega$ and the lateral boundary of $\Omega_T$ by  $\Sigma:=(0,T)\times\partial\Omega$ where $\displaystyle \partial \Omega:= \partial \omega\times \R$. We consider the following initial boundary value problem (IBVP) for the system of  wave equations:
\begin{equation}\begin{cases}\label{hprs}
   \mathcal{L}_{q}\overrightarrow {u}(t,x) = \overrightarrow {0}, &\quad (t,x)\in \Omega_T,\\
    \overrightarrow {u}(0,x)=\overrightarrow {\phi}(x),~~ \partial_t \overrightarrow {u}(0,x) = \overrightarrow {\psi}(x), &\quad x\in\Omega,\\
    \overrightarrow {u}(t,x) = \overrightarrow {f}(t,x), &\quad (t,x)\in  \Sigma,
\end{cases}
\end{equation}
where the operator $\mathcal{L}_{q}$ is described by
\begin{equation}\label{operator Lq}
\begin{split}
     \mathcal{L}_{q}\overrightarrow {u}(t,x) :&=\square \overrightarrow {u}(t,x) + q(t,x) \overrightarrow {u}(t,x)\\
     &={\displaystyle  \begin{bmatrix} \square u_1(t,x)+\sum_{j=1}^{n}q_{1j}(t,x)u_j(t,x) \\  \square u_2(t,x)+\sum_{j=1}^{n} q_{2j}(t,x)u_j(t,x) \\ \vdots \\ \square u_n(t,x)+\sum_{j=1}^{n} q_{nj}(t,x)u_j(t,x) \end{bmatrix}} ,\qquad (t,x)\in\Omega_T
\end{split}
\end{equation}
in which $\square:= \partial_t^2-\Delta_{x}$ denotes the standard wave operator and $$q(t,x)=((q_{ij}(t,x)))_{1\leq i,j\leq n} \ \  \mbox{with} \ \  q_{ij}\in {W}^{1,\infty}(\Omega_T),\ \ \mbox{for all}\ \  1\leq i,j\leq n,$$ represents a time-dependent matrix potential and $\displaystyle \overrightarrow {u}(t,x):=\left(u_1(t,x), u_2(t,x),\cdots, u_n(t,x)\right)^{T}$ represent the displacement vector. 

Following \cite{kian2013stability}, we introduce the following spaces 
\begin{align*}
H^s(\partial\Omega)&= H^s(\mathbb{R}_{x_3};L^2(\partial\omega))\cap L^2 (\mathbb{R}_{x_3};H^s(\partial\omega)),\\
H^{r,s}((0,T)\times X)&= H^r(0,T;L^2(X))\cap L^2(0,T;H^s(X)),
\end{align*}
where $X=\Omega$ or $X=\partial\Omega$, $s>0$ and $r\geq 0$. Suppose $\partial_{\tau}$ denotes the tangential derivative with respect to $\partial\omega$. Furthermore, we also set
{\begin{align*}
      L=\Big\{(\phi,\psi, f)\in H^1(\Omega)\times L^2(\Omega)\times H^{\frac{3}{2},\frac{3}{2}}( \Sigma): \phi|_{\partial\Omega}= f|_{t=0}, & ~\partial_{t}f,~
       \partial_{\tau}f,~ \partial_{x_3}f \in L^2\Big(\Sigma;d\sigma(x)\frac{dt}{t}\Big) \Big\}
\end{align*}
with $\|(\phi,\psi, f)\|_L$ defined by 
\begin{align*}
\|(\phi,\psi, f)\|^2_L=&
\|\phi\|^2_{H^1(\Omega)}  + \|\psi\|^2_{L^2(\Omega)} +
\|f\|^2_{H^{\frac{3}{2},\frac{3}{2}}(\Sigma)}\\
& +\int_{\Sigma}\frac{|\partial_tf|^2 +|\partial_\tau f|^2+|\partial_{x_3}f|^2 }{t}d\sigma(x)~dt.
\end{align*}}
In addition, throughout this paper, we will use the notation, $\mathbf{L}:= L\times \cdots \times L$ ($n$-times) and ${\HS}^s(X)$ represents the space of the vector-valued functions defined on $X$ with each of its component in $H^s(X)$. Similar notations will also be incorporated for the spaces $\CS^k(X)$, $\LS^2(X)$ as well.
{ To ensure the well-posedness of the IBVP \eqref{hprs}, we impose the initial data and boundary data $(\overrightarrow {\phi},\overrightarrow {\psi},\overrightarrow {f}) \in \mathbf{H}^{1} (\Omega)\times\mathbf{L}^{2}(\Omega)\times \mathbf{L}$ and assume the following  {\it global compatibility conditions} (see \cite[pp. 100]{lions1968problemes} for more details) 
\begin{align}\label{Global compatibility}
     \begin{cases}
        & \int_{\R}\int_{\R}\int_0^{\infty} \big|\partial_{x_2} \overrightarrow{\phi}|_{x_1=r} - \partial_{x_2}\overrightarrow{f}|_{t=r} \big|^2\frac{dr}{r} dx_2 dx_3 <\infty,\\
        & \int_{\R}\int_{\R}\int_0^{\infty} \big|\partial_{x_3} \overrightarrow{\phi}|_{x_1=r} - \partial_{x_3}\overrightarrow{f}|_{t=r} \big|^2\frac{dr}{r} dx_2 dx_3 <\infty
     \end{cases}
 \end{align}
whenever $\omega:=\{(x_1,x_2)\in\mathbb{R}^2:\ x_1>0\}$ (for details, kindly refer to Section \ref{sec:exis_uniq}).} As a consequence, using the  Theorem \ref{thm:recons_1} in Section \ref{sec:exis_uniq}, there exists a unique solution $\overrightarrow {u}$ of \eqref{hprs} satisfying the following: 
\begin{align*}
\overrightarrow {u}\in  \mathbf{C}^1 ([0,T]; \mathbf{ L}^{2}(\Omega)) \cap  \mathbf{C} ([0,T];  \mathbf{H}^{1}(\Omega))\, ~\quad\text{and}\quad \partial_\nu \overrightarrow {u}\in  \mathbf{L}^{2}(\Sigma),
\end{align*}
where  $\nu$ is the outward unit normal vector to $\partial\Omega$ and $\partial_\nu \overrightarrow {u}:=(\partial_\nu u_1,\partial_\nu u_2,\dots,\partial_\nu u_n)^T$ denotes the normal derivative of vector $\overrightarrow {u}$. 

Based on the existence and uniqueness of the solution to IBVP \eqref{hprs}, we define the input-output (IO) map 
 \begin{align*}
\Lambda_{{q}}: \mathbf{L}  \quad \rightarrow \quad \mathbf{H}^{1} (\Omega)\times\mathbf{L}^{2}(\Omega)\times \LS^2(\Sigma)
\end{align*}
associated to the system of wave equations \eqref{hprs} by 
\begin{equation}\label{input-output map}
\Lambda_{{q}}~\Big(\overrightarrow {\phi},\overrightarrow {\psi},\overrightarrow {f}\Big) := \Big(\overrightarrow {u} \big | _{t=T}, \partial_t \overrightarrow {u}\big |_{t=T}, \partial_\nu \overrightarrow {u}\big |_ {\Sigma}\Big),
\end{equation}
where $\overrightarrow {u}$ is the solution of \eqref{hprs}. Subsequently, we use the notations for its components as follows:
\begin{equation}\label{io map comp}
    \Lambda^{1}_{q}(\overrightarrow {\phi},\overrightarrow {\psi},\overrightarrow {f})=\overrightarrow {u} \big | _{t=T},\quad \Lambda^{2}_{q}(\overrightarrow {\phi},\overrightarrow {\psi},\overrightarrow {f})= \partial_t \overrightarrow {u}\big |_{t=T}, \quad \Lambda^{3}_{q}(\overrightarrow {\phi},\overrightarrow {\psi},\overrightarrow {f})= \partial_\nu \overrightarrow {u}\big |_ {\Sigma}.
\end{equation}

In the present paper, we consider the inverse problem of determining a time-dependent matrix potential $q(t,x)$ from the knowledge of IO map $\displaystyle \Lambda_q$. Our aim is to establish a stability estimate for the determination of $q$ from the given $\Lambda_q$.
We remark that the equation \eqref{hprs} describes the propagation of waves (for instance, electromagnetic waves or sound waves) along the axis of an infinite waveguide under the influence of zeroth-order potential $q(t,x)$. Moreover, equation \eqref{hprs} also appears in the model of transmitting light, signals, and sounds to long distances. 

Determination of ceofficients of wave operator are of great interest in recent times. 
In \cite{Bukhgeim_Klibanov,rakesh1988uniqueness}, the authors proved that the time independent potential can be determined uniquely from the knowledge of Dirichlet to Neumann (DN) map associated with a wave equation. The approach  used in \cite{rakesh1988uniqueness} was initially used   \cite{Sylvester and Uhlmann_calderon problem} for solving the Calder\'on problem.  The results in \cite{Bukhgeim_Klibanov,rakesh1988uniqueness} are further extended to the recovery of time-dependent potential in  \cite{Ramm and Rakesh,
ramm1991inverse} and for recovery of first order derivative  terms together with a potential coefficients are considered in  \cite{eskin2006new,eskin2007inverse,isakov1991inverse,Isakov-completeness,salazar2013determination}. All these works are related to recovery of coefficients either from the knowledge of DN map or IO map measured on full boundary. There have been several studies related to partial data uniqueness results as well. We refer to \cite{Kian_damping,
kian2017unique,
Kian and Oksanen,Krishnan and Vashisth,mishra2021determining} and references therein related to the determination of coefficients for wave operator from partial boundary measurements which are initially studied in \cite{Bukhgeim and Uhlmann} for the inverse problem related to  the Schr\"odinger equation. All the above mentioned works are related to unique determination of coefficients appearing in wave equation from boundary measurements. Next we will mention briefly the works related to deriving the stability estimate for coefficient determination problems for wave equation.
 Stefanov and Uhlmann \cite{stefanov1998stability} obtained the stability estimate from the DN-map considered on the lateral boundary. Under certain assumptions, H{\"o}lder stability estimate and \emph{log-log type} of stability estimate were derived in \cite {isakov1992stability} and \cite{bellassoued2009stability} respectively. With the help of \emph{light ray transform}, the log-log type stability estimate was also obtained in \cite{aicha2015stability} in case of time-dependent scalar potential from the knowledge of IO-map associated with a wave equation. 
For more works related to determining the stability estimates for coefficients appearing in a single wave equation from the boundary measurements, one can refer to \cite{bellassoued2004uniqueness, bellassoued2006lipschitz, bellassoued2008stability,Bellassoued and Ibtissem_Optimal,
Bellassoued and Ibtissem JMAA, Bellassoued and Rassas,kian2016stability,Salazar stability,Soumen} and references therein.
In the case of an infinite cylindrical domain, the literature on stability estimates is limited. For instance, Bellassoued et al. \cite{bellassoued2018inverse} and Kian et al. \cite{kian2015holder} studied an inverse problems for magnetic Schr\"odinger equation and established the H{\"o}lder stability estimates.  In \cite{bellassoued2016inverse}, the authors investigated the inverse problem of determining the time-independent scalar potential of the dynamic Schr{\"o}dinger equation in an infinite cylindrical domain from partial measurement of the solution on the boundary.
For related work in the case of an infinite cylindrical domain, one may refer to \cite{kian2014carleman,salo2006complex} and references therein.

To the best of our knowledge, the study of recovery of time-dependent matrix potential in the literature is very limited. However, in the case of time-independent matrix potential, the authors in \cite{avdonin1992boundary} proved that it can be recovered from the boundary measurements. Furthermore, Eskin and Ralston in \cite{eskin2003inverse} considered the problem of determining zeroth order and first order matrix-valued perturbations in evolution equations and proved uniqueness from the full boundary measurements. Khanfer et al. \cite{khanfer2019inverse} also demonstrated uniqueness in the case of time-dependent matrix potential under certain assumptions on the potential and when the spatial dimension is one.
Considering a bounded domain, Mishra and Vashisth in \cite{mishra2021determining} demonstrated the uniqueness of the time-dependent matrix potential for a wave equation from the partial boundary data. In this paper, we consider the stability estimate in the determination of time-dependent matrix potential over an infinite waveguide from the knowledge of the input-output map. We remark that the estimate is obtained without the assumption of the behaviour of the matrix potential outside a compact set.

The paper is organized as follows. In Section \ref{sec:exis_uniq}, we prove the well-posedness of the forward problem \eqref{hprs}. In Section \ref{sec:main_res}, we state the main results of the stability estimate. In Section \ref{sec:GO_soln}, we develop geometric optics (GO) solutions for our problem. In Section \ref{sec:stab_esti}, we derive an integral identity for our problem. Using the GO solutions and light ray transform along with the integral identity, we demonstrate the stability estimate. Moreover, under certain assumptions in Section \ref{sec:stab_esti}, we extend this result to the same inverse problem with measurements on a bounded subset of the lateral boundary.
\section{Existence and Uniqueness}\label{sec:exis_uniq}
In this section, we establish the existence and uniqueness of the IBVP \eqref{hprs}. More precisely, we prove the following theorem:
\begin{theorem}\label{thm:recons_1}
Let $\displaystyle q(t,x):=((q_{ij}(t,x)))_{1\leq i,j\leq n}$ with  $q_{ij}\in W^{1,\infty}(\Omega_T)$ for all $1\leq i,j\leq n$, be a time-dependent matrix potential. Suppose the initial and boundary data $(\overrightarrow {\phi},~\overrightarrow {\psi}, \overrightarrow {f}) \in \mathbf{L}$. Then, there exists a unique solution
\begin{align*}
\overrightarrow {u}\in  \mathbf{C}^1 ([0,T]; \mathbf{ L}^{2}(\Omega)) \cap  \mathbf{C} ([0,T];  \mathbf{H}^{1}(\Omega))
\end{align*}
of the IBVP \eqref{hprs} such that $\partial_\nu \overrightarrow {u}\in  \mathbf{L}^{2}(\Sigma)$.
Moreover, it satisfies the energy estimate 
\begin{align}\label{wellposed_esti}
\|\partial_\nu \overrightarrow {u}\|_{\mathbf{L}^{2}(\Sigma)}
+\|\overrightarrow {u}\|_{\CS([0,T];\HS^1(\Omega))}
+\|\overrightarrow {u}\|_{\CS^1([0,T];\LS^2(\Omega))}
\leq C \|(\overrightarrow {\phi},
\overrightarrow{\psi},\overrightarrow {f})\|_{\LS},
\end{align}
where the constant $C= C(\Omega, T, \|q\|_{W^{1,\infty}(\Omega_T)})>0$.
\end{theorem}
 In the case of bounded domain $\Omega$, the Theorem \ref{thm:recons_1} has been proved in \cite{mishra2021determining}. However, the same can not be applicable to the unbounded infinite waveguide. Nevertheless, we prove the Theorem \ref{thm:recons_1} by following the approach used in \cite{kian2013stability} where a well-posedness result is proved for a single wave equation. We require the following lemma to prove Theorem \ref{thm:recons_1}
\begin{lemma}\label{lem:lifting_sob}
    For all $(\overrightarrow {\phi},\overrightarrow {\psi},\overrightarrow {f})\in \mathbf{L}$, there exists 
    $$\overrightarrow {w} = \overrightarrow {w}[\overrightarrow {\phi},\overrightarrow {\psi},\overrightarrow {f}] = (w_{1}, w_{2},\dots,w_{n})^T\in \HS^{2,2}(\Omega_T)$$ satisfying 
    \begin{equation}\label{w(f)}
        \begin{cases}
            \overrightarrow {w}(0,\cdot)=\overrightarrow {\phi}, \quad \partial_t \overrightarrow {w}(0,\cdot)=\overrightarrow {\psi}, &\text{in}~\Omega,\\
            \partial_{\nu}\overrightarrow {w}=\overrightarrow {0},\quad \overrightarrow {w}=\overrightarrow {f}, &\text{on}~ \Sigma,
        \end{cases}
    \end{equation}
    and 
    \begin{equation}\label{estimat for w[f]}
        \|\overrightarrow {w}\|_{\HS^{2,2}(\Omega_T)}\leq C \|(\overrightarrow {\phi},\overrightarrow {\psi},\overrightarrow {f})\|_{\mathbf{L}}.
    \end{equation}
\end{lemma}
The Lemma \ref{lem:lifting_sob} can be proved if we show that for every $i=1,2,\dots,n$ and for each $(\phi_i,\psi_i,f_i)\in L$, there exists $w_i\in H^{2,2}(\Omega_T)$ satisfying 
    \begin{equation}\label{wi(f)}
         \begin{cases}
            w_i(0,x)=\phi_i, \quad \partial_t  w_i(0,x)=\psi_i, &\text{in}~\Omega,\\
            \partial_{\nu} w_i=0,\quad  w_i= f_i, &\text{on}~ \Sigma,
        \end{cases}
    \end{equation}
     and 
    \begin{equation}\label{wi(f) estimate}
        \| w_i\|_{H^{2,2}(\Omega_T)}\leq C \|(\phi_i,\psi_i,f_i)\|_{L}.
    \end{equation}
{ Our aim is to show Lemma \ref{lem:lifting_sob} related to the lifting of Sobolev spaces using the arguments related to local coordinates as used in  \cite[Chapter 1]{lions2012non} (see pages $38-40$) for proving the trace theorem for $L^2$-Sobolev spaces. 
In this regard, we first prove the result when  the domain $\omega$ and its boundary $\partial\omega$ are given by $\R^2_{+}=\{(x_1,x_2): x_1>0\}$ and $\R$ respectively.} Subsequently, $\Omega$ and $\partial\Omega$ will be replaced by $\R^3_{+}$ and $\R^2$ respectively. We assume $T=\infty$. Finally, $\Omega_T$ and $\Sigma$ can be replaced by $((0,\infty)\times\R^3_{+})$ and $((0,\infty)\times\R^2)$ respectively. Our goal is to ensure the existence of $w_i$ satisfying \eqref{wi(f)} and \eqref{wi(f) estimate}. For convenience, we omit the subscript $i$ for the following analysis. 

We consider the space $Z_1$ consisting of $(u_0,u_1,g_0,g_1)$ such that
\begin{align}
    u_0\in H^{\frac{3}{2}}(\R^3_{+}), \quad u_1\in H^{\frac{1}{2}}(\R^3_{+}), \quad g_0\in H^{\frac{3}{2},\frac{3}{2}}((0,\infty)\times\R^2), \quad g_1\in H^{\frac{1}{2},\frac{1}{2}}((0,\infty)\times\R^2)
\end{align}
with the global compatibility conditions 
 \begin{align}
     \begin{cases}
         & u_0|_{\partial\Omega} = g_0|_{t=0},\\
        & \int_{\R}\int_{\R}\int_0^{\infty} \big|\partial_{x_2} u_0|_{x_1=r} - \partial_{x_2}g_0|_{t=r} \big|^2\frac{dr}{r} dx_2 dx_3 <\infty,\\
        & \int_{\R}\int_{\R}\int_0^{\infty} \big|\partial_{x_3} u_0|_{x_1=r} - \partial_{x_3}g_0|_{t=r} \big|^2\frac{dr}{r} dx_2 dx_3 <\infty.
     \end{cases}
 \end{align}
Now let us consider the Hilbert space
{\begin{align*}
   \ Z_2 := \Big\{(\phi,\psi,f): \phi|_{\partial\Omega} = f|_{t=0},&~ t^{-\frac{1}{2}}\partial_t f, ~ t^{-\frac{1}{2}}\partial_{x_2}\phi|_{x_1=t},\\
   & t^{-\frac{1}{2}}\left(\nabla_{(x_2,x_3)}f - \nabla_{(x_2,x_3)}\phi\big|_{x_1=t}\right) \in L^2((0,\infty)\times\R^2)\Big\}
 \end{align*}
 and the associated norm is defined by
 \begin{align*}
     \|(\phi,\psi,f)\|^2_{Z_2} := & \|\phi\|^2_{H^{\frac{3}{2}}(\R^3_{+})} + \|\psi\|^2_{H^{\frac{1}{2}}(\R^3_{+})} + \|f\|^2_{H^{\frac{3}{2},\frac{3}{2}}((0,\infty)\times\R^2)} + \|t^{-\frac{1}{2}}\partial_t f\|^2_{L^2((0,\infty)\times\R^2)}\\
     &  + \|t^{-\frac{1}{2}}\left(\nabla_{(x_2,x_3)}f - \nabla_{(x_2,x_3)}\phi|_{x_1=t}\right)\|^2_{L^2((0,\infty)\times\R^2)}.
 \end{align*}}
 With the above set-up, we shall prove Lemma \ref{lem:lifting_sob}. 
\begin{proof}[Proof of Lemma \ref{lem:lifting_sob}]
    Let us define an operator $S: H^{2,2}((0,\infty)\times\R^3_{+})\rightarrow Z_1$ as
    \begin{align*}
        S(w):= (w|_{t=0},\partial_t w|_{t=0}, w|_{x_1=0}, \partial_{x_1}w|_{x_1=0}).
    \end{align*}
Then the linear operator $S$ is continuous and onto (refer to Theorem $2.3$ on page 18 in \cite[Chapter 4]{lions1968problemes}). Afterwards, we show that for any $(\phi,\psi,f)\in Z_2$, there exists $w\in H^{2,2}((0,\infty)\times\R^3_{+})$ such that 
\begin{align}\label{temp:lift_1}
    (w|_{t=0},\partial_t w|_{t=0}, w|_{x_1=0}, \partial_{x_1}w|_{x_1=0}) = (\phi,\psi,f,0)
\end{align}
and the following estimate holds
 \begin{align}\label{temp:lift_2}
     \|w\|_{H^{2,2}((0,\infty)\times\R^3_{+})} \leq C \|(\phi,\psi,f)\|^2_{Z_2}.
 \end{align}
Let $(u_0,u_1,g_0,g_1)=(\phi,\psi,f,0)$.   Then it is observed that $(u_0,u_1,g_0,g_1)\in Z_1$ if and only if $(\phi,\psi,f)\in Z_2$. Furthermore, we consider the space 
 \begin{align*}
     Z_3 := \{w\in H^{2,2}((0,\infty)\times\R^3_{+}): \partial_{x_1}w|_{x_1=0}=0\}.
 \end{align*}
We observe that the restriction of the operator $S$ to $Z_3$ is continuous and onto from $Z_3$ to {$Z_2\times\{0\}$}. As a consequence, we deduce the operator 
 \begin{align*}
     S_1: w\mapsto (w|_{t=0},\partial_t w|_{t=0}, w|_{x_1=0})
 \end{align*}
is continuous and onto from the space $Z_3$ to the Hilbert space $Z_2$. Hence there exists a bounded operator $S_2: Z_2\rightarrow Z_3$ such that $S_1S_2=I$. For $w=S_2(\phi,\psi,f)$, \eqref{temp:lift_1} holds and the estimate \eqref{temp:lift_2} is satisfied. { This completes the proof when the domain $\omega$ and its boundary $\partial\omega$ are given by $\R^2_{+}=\{(x_1,x_2): x_1>0\}$ and $\R$ respectively. Now since $\omega\subset\mathbb{R}^2$ be an open set with smooth boundary therefore there exist $\{U_j\}_{1\leq j\leq  N}$ a  family of open subsets of $\mathbb{R}^2$ such that $\displaystyle \partial\omega\subset\cup_{1\leq j \leq N}U_j$ and smooth maps $\displaystyle \{\eta_j\}_{1\leq j\leq N}$ such that for each $1\leq j\leq N$, the map $\displaystyle \eta_j: U_j\rightarrow \mathcal{Q}:=\{(y_1,y_2):\ \lvert y_1\rvert <1,\ \&\ \lvert y_2\rvert<1\}$ are bijective and satisfies the  following properties
\begin{align*}
    \begin{aligned}
        \eta_j(U_j\cap \omega)=\{(y_1,y_2)\in\mathcal{Q}:\ y_1>0\}:=\mathcal{Q}_j^{*},\ \eta_j(U_j\cap \partial\omega)=\{ (y_1,y_2)\in\mathcal{Q}:\ y_1=0\}. 
    \end{aligned}
\end{align*}
Since the composition of a Sobolev function with a smooth diffeomorphism is again a Sobolev function in the respective domain, therefore we have the composition function $f\circ\eta_j^{-1}$ lies in a Sobolev space whenever $f$ is a Sobolev function.  
Finally after choosing an open set $U_0\subset\mathbb{R}^2$ such that $\displaystyle \overline{\omega}\subset\cup_{j=0}^{N}U_j$ and  a partition of unity  subordinate to the cover $\{U_j\}_{0\leq j\leq N}$, we can replace the space  $L$  along with the global compatibility conditions given in \eqref{Global compatibility}  by $Z_2$ 
(see \cite[Proposition 3.3]{lions1968problemes} for details)}. 
Subsequently, we obtain that $w_i\in H^{2,2}(\Omega_T)$ such that \eqref{wi(f)} and \eqref{wi(f) estimate} are satisfied.
Hence the result follows.
\end{proof}

\begin{proof}[Proof of Theorem \ref{thm:recons_1}]
Thanks to Lemma \ref{lem:lifting_sob}, we split $\overrightarrow {u}$ into two parts $\overrightarrow {u}=\overrightarrow {v}+\overrightarrow {w}$ with $\overrightarrow {w}=\overrightarrow {w}[\overrightarrow {\phi},\overrightarrow {\psi},\overrightarrow {f}]\in \mathbf{H}^{2,2}(\Omega_T)$ satisfying \eqref{w(f)}-\eqref{estimat for w[f]} and $\overrightarrow {v}$ satisfies the following IBVP:
\begin{equation}\label{equtn for v}
    \begin{cases}
        \partial_t^2 \overrightarrow {v}(t,x) - \Delta \overrightarrow {v}(t,x) + q(t,x) \overrightarrow {v}(t,x) = \overrightarrow {F}(t,x), &\quad (t,x)\in \Omega_T,\\
    \overrightarrow {v}(0,x)=\overrightarrow {0},~~ \partial_t \overrightarrow {v}(0,x) = \overrightarrow {0}, &\quad x\in\Omega,\\
    \overrightarrow {v}(t,x) = \overrightarrow {0}, &\quad  (t,x)\in  \Sigma  
    \end{cases}
\end{equation}
 with $\overrightarrow {F}(t,x)=-\big(\partial_t^2 \overrightarrow {w}(t,x) - \Delta \overrightarrow {w}(t,x) + q(t,x) \overrightarrow {w}(t,x)\big)$.
Since $q\in W^{1,\infty}(\Omega_T)$ and $\overrightarrow{w}\in \HS^{2,2}(\Omega_T)$, we have $\overrightarrow {F}(t,x)\in \mathbf{L}^2(\Omega_T)$. 
To demonstrate the well-posedness of IBVP \eqref{equtn for v}, we follow the arguments from \cite{lions2012non, lions1968problemes, mishra2021determining}. Let us define a time-dependent bilinear form
 \begin{align}\label{temp:wellposed_1}
     a(t;\overrightarrow {v},\overrightarrow {w}): = \int_{\Omega} \nabla \overrightarrow {v}(x)\cdot \overline{\nabla \overrightarrow {w}(x)}~dx + \int_{\Omega} q(t,x)\overrightarrow {v}(x)\cdot\overline{\overrightarrow {w}(x)}~dx, \quad \overrightarrow {v},\overrightarrow {w}\in \HS^1_0(\Omega).
 \end{align}
Since $q\in W^{1,\infty}(\Omega_T)$ and $\overrightarrow {v}, \overrightarrow {w}$ are time-independent, we have the following bound
 \begin{align}\label{temp:wellposed_2}
     |a(t;\overrightarrow {v},\overrightarrow {w})| \leq \alpha\|\overrightarrow {v}\|_{\HS^1_0(\Omega)}
     \|\overrightarrow {w}\|_{\HS^1_0(\Omega)},
 \end{align}
 where we have used the Cauchy-Schwartz inequality and the positive constant $\alpha$ is independent of $\overrightarrow {v}$ and $\overrightarrow {w}$. Furthermore, we observe that
\begin{align*}
    |a(t;\overrightarrow {w},\overrightarrow {w})| & = \Big|\int_{\Omega} \Big(|\nabla \overrightarrow {w}(x)|^2 + q(t,x)\overrightarrow {w}(x)\cdot\overline{\overrightarrow {w}(x)}\Big)~dx \Big|\\
    & \geq \|\nabla \overrightarrow {w}\|^2_{\LS^2(\Omega)} - \|q\|_{L^{\infty}(\Omega_T)}\|\overrightarrow {w}\|^2_{\LS^2(\Omega)}
\end{align*}
and consequently, we obtain the following estimate
\begin{align}\label{temp:wellposed_3}
    \beta \|\overrightarrow {w}\|^2_{\HS^1(\Omega)} \leq |a(t;\overrightarrow {w},\overrightarrow {w})|+ \gamma \|\overrightarrow {w}\|^2_{\LS^2(\Omega)}, \quad \overrightarrow {w}\in \HS^1_0(\Omega),~t\in (0,T),
\end{align}
for some constants $\beta>0$ and $\gamma=1+\|q\|_{L^{\infty}(\Omega_T)}$.
From \eqref{temp:wellposed_1}-\eqref{temp:wellposed_3}, we conclude that the map $t\mapsto a(t;\overrightarrow {v},\overrightarrow {w})$ is a continuous bilinear for all $\overrightarrow {v},\overrightarrow {w}\in \HS^1_0(\Omega)$ and $t\in (0,T)$. Since the principal part of $a(t;\cdot,\cdot)$ is anti-symmetric,
applying \cite[Theorem 8.1]{lions2012non} we ensure that there exists
a unique solution $\overrightarrow {v}\in \CS(0,T;\HS^1(\Omega))\cap \CS^1(0,T;\LS^2(\Omega))$ of \eqref{equtn for v}. In addition, there holds
\begin{align}\label{temp:lift_esti_v}
    \|\overrightarrow {v}\|_{\CS([0,T];\HS^1(\Omega))} + \|\overrightarrow {v}\|_{\CS^1([0,T];\LS^2(\Omega))}
    \leq C \|(\overrightarrow {\phi},\overrightarrow {\psi},\overrightarrow {f})\|_{\LS}.
\end{align}
From the Lemma \ref{lem:lifting_sob}, there exists $\overrightarrow{w}\in \HS^{2,2}(\Omega_T)$ such that \eqref{w(f)} holds. Moreover, using the inclusion (\cite[Chapter 1, Theorem 3.1]{lions2012non})
\begin{align*}
    \HS^{2,2}(\Omega_T) \xhookrightarrow{} \CS(0,T;\HS^{\frac{3}{2}}(\Omega)) \cap \CS^1(0,T;\HS^{\frac{1}{2}}(\Omega))\xhookrightarrow{} \CS(0,T;\HS^{1}(\Omega)) \cap \CS^1(0,T;\LS^{2}(\Omega)),
\end{align*}
we have the following estimate
\begin{align}\label{temp:lift_esti_w}
    \|\overrightarrow{w}\|_{\CS(0,T;\HS^1(\Omega))} + \|\overrightarrow{w}\|_{\CS^1(0,T;\LS^2(\Omega))} \leq C\|\overrightarrow{w}\|_{\HS^{2,2}(\Omega_T)} \leq C \|(\overrightarrow {\phi},\overrightarrow {\psi},\overrightarrow {f})\|_{\LS}.
\end{align}
Combining the estimates \eqref{temp:lift_esti_v} and \eqref{temp:lift_esti_w} together, we have 
 \begin{align*}
     \overrightarrow{u} \in \CS(0,T;\HS^{1}(\Omega)) \cap \CS^1(0,T;\LS^{2}(\Omega))
 \end{align*}
along with the estimate
\begin{align}
\|\overrightarrow{u}\|_{\CS(0,T;\HS^1(\Omega))} + \|\overrightarrow{u}\|_{\CS^1(0,T;\LS^2(\Omega))} \leq C \|(\overrightarrow {\phi},\overrightarrow {\psi},\overrightarrow {f})\|_{\LS}.
\end{align}

Afterwards, we show that $\partial_{\nu}\overrightarrow {v}\in \mathbf{L}^2(\Sigma)$ and
\begin{align}\label{temp:esti_v_nrml}
   \|\partial_{\nu}\overrightarrow {v}\|_{\LS^2(\Sigma)} 
    \leq C\|(\overrightarrow {\phi},\overrightarrow {\psi},\overrightarrow {f})\|_{\LS}.
\end{align}
To do this, for each $i\in \{1,2,\dots,n\}$, let us define $\overrightarrow {K}^{(i)}=(\delta_{ir})_{r=1}^n$. As a consequence, we have
\begin{align}\label{well_posed_temp_1}
\begin{aligned}
    \begin{cases}
        \left(\partial_t^2v_i-\Delta v_i+\sum_{j=1}^{n} q_{ij}v_j\right)(t,x)= \overrightarrow {F}(t,x)\cdot \overrightarrow {K}^{(i)},\ \    (t,x)\in \Omega_T,\\
        v_i(0,x)=0,~~ \partial_t v_i(0,x) =0,  \qquad x\in\Omega,\\
    v_i(t,x) =0,\qquad   (t,x)\in\Sigma,
\end{cases}
\end{aligned}
\end{align}
for any $1\leq i\leq n$.
Let $\nu_1$ be the unit outward normal vector to $\partial\omega$ and let us consider the function $\gamma\in{C}^\infty(\overline{\Omega},\mathbb{R}^3)$, which is defined by $$\gamma(x',x_3)=(\gamma_1(x'),0),\qquad x'\in \omega,~ x_3\in\mathbb{R},$$
where $\gamma_1\in{C}^\infty(\overline{\omega},\mathbb{R}^2)$ satisfies $\gamma_1=\nu_1$\, on $\partial\omega$. Hence we obtain 
$\gamma=\nu$ on $\partial\Omega$. 
We multiply the equation \eqref{well_posed_temp_1} by $\gamma\cdot\nabla v_i$ and integrate over $\Omega_T$ to obtain 
\begin{align}
    \mathcal{E}_1 + \mathcal{E}_2:&=\int_{\Omega_T} \big(\partial_t^2v_i  -\Delta v_i\big)\big(\gamma\cdot\nabla v_i\big)\,dx\,dt  \nonumber\\
        &= - \int_{\Omega_T}\big(\gamma\cdot\nabla v_i\big)\sum_{j=1}^{n} q_{ij}v_j\,dx\,dt +\int_{\Omega_T}\big(\overrightarrow {F}\cdot \overrightarrow {K}^{(i)}\big)\big(\gamma\cdot\nabla v_i\big)\,dx\,dt \nonumber\\
        & \leq  C \bigg(\int_{\Omega_T} |\overrightarrow {v}|^2\,dx\,dt + \int_{\Omega_T} |\nabla v_i|^2\,dx\,dt +\int_{\Omega_T} {|\big(\overrightarrow {F}\cdot \overrightarrow {K}^{(i)}\big)|^2}\,dx\,dt \bigg) \nonumber\\
        & \leq C \Big(\|\overrightarrow {w}\|^2_{\mathbf{H}^{2,2}(\Omega_T)}+\int^T_0(\|\overrightarrow {v}\|^2_{\mathbf{H}^1_0(\Omega)}+\|v_i\|^2_{H^1_0(\Omega)})\,dt\Big) \nonumber\\
        & \leq C \Big(\|(\overrightarrow {\phi},\overrightarrow {\psi},\overrightarrow {f})\|^2_{\LS}+\|\overrightarrow {v}\|^2_{\mathcal{C}([0,T];\mathbf{H}^1_0(\Omega))} \Big) \label{energy estimate component wise}.
\end{align}
 Using the integration by parts, we get
\begin{equation*}
    \begin{split}
        \mathcal{E}_1 & = \int_{\Omega_T} \partial_t^2v_i\big(\gamma\cdot\nabla v_i\big)\,dx\,dt \\
         & = -\int_{\Omega_T}\partial_t v_i\big(\gamma\cdot\nabla \partial_tv_i\big)\,dx\,dt+\int_{\Omega}\partial_t v_i(T,x)\big(\gamma\cdot\nabla v_i(T,x)\big)\,dx\\
         & \quad -\int_{\Omega}\partial_t v_i(0,x)\big(\gamma\cdot\nabla v_i(0,x)\big)\,dx\,dt\\
         &=\int_{\Omega}\partial_t v_i(T,x)\big(\gamma\cdot\nabla v_i(T,x)\big)\,dx-\frac{1}{2}\int_{\Omega_T}\big(\gamma\cdot\nabla (\partial_tv_i)^2\big)\,dx\,dt.
    \end{split}
\end{equation*}
With the help of Green's formula in $x'\in\omega$ and the identity $\gamma(x)\cdot\nabla (\partial_tv_i)^2=\gamma_1(x')\cdot\nabla_{x'} (\partial_tv_i)^2$, we obtain
\begin{equation*}
    \begin{split}
        \int^T_0\int_{\Omega} & \gamma\cdot\nabla (\partial_tv_i)^2\,dx\,dt =\int^T_0\int_{\mathbb{R}}\int_{\omega}\gamma_1\cdot\nabla_{x'} (\partial_tv_i)^2\,dx'\,dx_3\,dt\\
        &=-\int^T_0\int_{\mathbb{R}}\int_{\omega} (\nabla\cdot\gamma )(\partial_tv_i)^2\,dx'\,dx_3\,dt +\int^T_0\int_{\mathbb{R}}\int_{\partial\omega}(\partial_tv_i)^2 \,dS_{x'}\,dx_3\,dt\\
        &=-\int_{\Omega_T}(\nabla\cdot\gamma )(\partial_tv_i)^2\,dx'\,dx_3\,dt,
    \end{split}
\end{equation*}
where we have used the fact that $v_i\big|_{\partial\Omega}=0$. Hence we end up with
\begin{equation}\label{Equality involving Backward time}
   \begin{split}
        \mathcal{E}_1= \int_{\Omega}\partial_t v_i(T,x) \big(\gamma\cdot\nabla v_i(T,x)\big)\,dx +\frac{1}{2}\int_{\Omega_T}(\nabla\cdot\gamma )(\partial_tv_i)^2\,dx\,dt.
   \end{split}
\end{equation}
Next, we focus on $\mathcal{E}_2$. We observe that
\begin{equation*}
   \mathcal{E}_2= -\int_{\Omega_T} \Delta v_i \big(\gamma\cdot\nabla v_i\big)\,dx\,dt = 
   -\int_{\Omega_T} \Delta_{x'} v_i \big(\gamma\cdot\nabla v_i\big)\,dx\,dt 
   -\int_{\Omega_T} \partial^2_{x_3}v_i \big(\gamma\cdot\nabla v_i\big)\,dx\,dt.
\end{equation*}
Again applying Green's formula in $x'\in\omega$, we find
\begin{equation}
 \begin{split}
  -\int_{\Omega_T} \Delta_{x'} v_i \big(\gamma\cdot\nabla v_i\big)\,dx\,dt & = -\int^T_0\int_{\mathbb{R}}\int_{\partial\omega}|\partial_{\nu}v_i|^2\,dS_{x'}\,dx_3\,dt \nonumber\\
&\quad+\int^T_0\int_{\mathbb{R}}\int_{\omega}\nabla_{x'}v_i\cdot\nabla_{x'}\big(\gamma\cdot\nabla v_i\big)\,dx\,dt,
    \end{split}
\end{equation}
and using integration by parts in $x_3\in\mathbb{R}$, we obtain
\begin{equation*}
  -\int_{\Omega_T} \partial^2_{x_3}v_i \big(\gamma\cdot\nabla v_i\big)\,dx\,dt = \int_{\Omega_T} \partial_{x_3}v_i \partial_{x_3}\big(\gamma\cdot\nabla v_i\big)\,dx\,dt. 
\end{equation*}
As a consequence, $\mathcal{E}_2$ reduces to
\begin{equation*}
     \mathcal{E}_2 = -\int_{\Sigma}|\partial_{\nu}v_i|^2\,dS_{x'}\,dx_3\,dt
     +\int_{\Omega_T}\nabla_{x}v_i\cdot\nabla_{x}\big(\gamma\cdot\nabla v_i\big)\,dx\,dt.
\end{equation*}
Furthermore, we use the following identity for any $\gamma=(\gamma_1,\gamma_2,\gamma_3)^T\in \R^3$ and $H=(\partial_{x_j}\gamma_i)_{1\leq i,\,j\leq 3}$,
$$\nabla_{x}v_i\cdot\nabla_{x}\big(\gamma\cdot\nabla v_i\big)=\big(H\nabla v_i\big)\cdot\nabla v_i+\frac{1}{2}\gamma\cdot\nabla(|\nabla v_i|^2).$$ 
Subsequently, $\mathcal{E}_2$ transforms into
\begin{equation}\label{integral identity over Omega and Partial omega}
     \begin{split}
         \mathcal{E}_2 & = -\int_{\Sigma}|\partial_{\nu}v_i|^2\,dS_{x'}\,dx_3\,dt+\int_{\Omega_T}\big(H\nabla v_i\big)\cdot\nabla v_i\,dx\,dt
         +\frac{1}{2}\int_{\Omega_T}\gamma\cdot\nabla(|\nabla v_i|^2)\,dx\,dt.
     \end{split}
\end{equation}
Again, applying the Green's formula in $x'\in\omega$ implies
\begin{equation}\label{Green's formula in x"}
    \begin{split}
        \int_{\omega}\gamma\cdot\nabla(|\nabla v_i|^2)\,dx'&=\int_{\omega}\gamma_1\cdot\nabla_{x'}(|\nabla v_i|^2)\,dx'\\
        &= \int_{\partial\omega}(|\nabla v_i|^2)\,dS_{x'}-\int_{\omega}(\nabla\cdot\gamma)(|\nabla v_i|^2)\,dx'.
    \end{split}
\end{equation}
Using the fact that $v_i\big|_{\partial\Omega}=0$, we have $|\nabla v_i|^2=|\partial_{\nu}v_i|^2$ on $\Sigma$. It follows from \eqref{Green's formula in x"} that
  \begin{equation}\label{Green func integral extension}
      \begin{split}
        \int_{\Omega_T}\gamma\cdot\nabla(|\nabla v_i|^2)\,dx\,dt=\int_{\Sigma}|\partial_{\nu}v_i|^2\,dS_{x'}\,dx_3\,dt-\int_{\Omega_T}(\nabla\cdot\gamma)(|\nabla v_i|^2)\,dx\, dt.
        \end{split}
  \end{equation}
  Substituting \eqref{Green func integral extension} in \eqref{integral identity over Omega and Partial omega} we get,
  \begin{equation}\label{Equality with normal derivative}
      \begin{split}
          \mathcal{E}_2&= -\frac{1}{2}\int_{\Sigma}|\partial_{\nu}v_i|^2 \,dS_{x'}\,dx_3\,dt +\int_{\Omega_T}\big(H\nabla v_i\big)\cdot\nabla v_i\,dx\,dt
           -\frac{1}{2}\int_{\Omega_T}(\nabla\cdot\gamma)(|\nabla v_i|^2)\,dx\, dt.
      \end{split}
  \end{equation}
  Combining \eqref{energy estimate component wise}, \eqref{Equality involving Backward time} and \eqref{Equality with normal derivative} together, we get
  \begin{equation*}
      \begin{split}
     \int_{\Sigma}|\partial_{\nu}v_i|^2\,dS_{x'} \,dx_3\,dt
        & =2 \int_{\Omega_T}\big(H\nabla v_i\big)\cdot\nabla v_i\,dx\,dt - \int_{\Omega_T}(\nabla\cdot\gamma)(|\nabla v_i|^2)\,dx\, dt\\
        &\quad + 2\int_{\Omega}\partial_t v_i(T,x)\big(\gamma\cdot\nabla v_i(T,x)\big)\,dx +\int_{\Omega_T}(\nabla\cdot\gamma )(\partial_tv_i)^2\,dx'\,dx_3\,dt\\
        & \quad + \int_{\Omega_T}\big(\gamma\cdot\nabla v_i\big)\sum_{j=1}^{n} q_{ij}v_j\,dx\,dt -\int_{\Omega_T}\big(\overrightarrow {F}\cdot \overrightarrow {K}^{(i)}\big)\big(\gamma\cdot\nabla v_i\big)\,dx\,dt 
      \end{split}
  \end{equation*}
which further implies that
\begin{equation*}
\begin{split}
\|\partial_{\nu}v_i\|^2_{L^2(\Sigma)}\leq C \Big( \|(\overrightarrow {\phi},\overrightarrow {\psi},\overrightarrow {f})\|^2_{\LS}+\|\overrightarrow {v}\|^2_{\CS([0,T];\HS^1_0(\Omega))}+\|\overrightarrow {v}\|^2_{\CS^1([0,T];\LS^2(\Omega))}\Big).
\end{split}
\end{equation*}
Since the above estimate holds for each $i=1,2,\dots,n$, we finally end up with
\begin{equation*}
      \begin{split}
          \|\partial_\nu \overrightarrow {v}\|^2_{\mathbf{L}^{2}(\Sigma)}\leq C \Big( \|(\overrightarrow {\phi},\overrightarrow {\psi},\overrightarrow {f})\|^2_{\LS}+\|\overrightarrow {v}\|^2_{\CS([0,T];\HS^1_0(\Omega))}+\|\overrightarrow {v}\|^2_{\CS^1([0,T];\LS^2(\Omega))}\Big)
      \end{split}
  \end{equation*}
and consequently, the estimate \eqref{temp:esti_v_nrml} is established.
Moreover, using the fact that $ \|\partial_{\nu}\overrightarrow {u}\|_{\LS^2(\Sigma)}=  \|\partial_{\nu}\overrightarrow {v}\|_{\LS^2(\Sigma)}$, we have
\begin{align*}
    \|\partial_\nu \overrightarrow {u}\|_{\mathbf{L}^{2}(\Sigma)}\leq C \|(\overrightarrow {\phi},\overrightarrow {\psi},\overrightarrow {f})\|_{\LS}.
\end{align*}
Hence the result follows.
\end{proof}
\begin{remark}
Incorporating \eqref{wellposed_esti}, we find that the input-output map $\Lambda_{{q}}$ defined in \eqref{input-output map} is continuous from $\mathbf{L}$ to $\mathbf{H}^{1} (\Omega)\times\mathbf{L}^{2}(\Omega)\times \LS^2(\Sigma)$.
\end{remark}
\section{statement of the main results}\label{sec:main_res}
The main results of the paper are as follows:
\begin{theorem}\label{main result}
Let $\displaystyle q^{(k)}:=((q^{(k)}_{ij}))_{1\leq i,j\leq n}$ for $k=1,2$  be two sets of matrix potentials with  with $q^{(1)}_{ij},q^{(2)}_{ij}\in W^{1,\infty}(\Omega_T)$ for all $1\leq i,j\leq n$ and $\|q^{(k)}\|\leq M,$ for $k=1,2$.
Let $\overrightarrow {u}^{(k)}$ be the solution of \eqref{hprs} corresponding to the matrix potential $q= q^{(k)}$ and $\Lambda_{q^{(k)}}$ be the given IO map defined by \eqref{input-output map} corresponding to $\overrightarrow {u}^{(k)}$. Then, the following stability estimate holds for $T>Diam(\omega)$,
\begin{align}\label{main_stability_estimate_for_q}
    \big\|q^{(1)}-q^{(2)} \big\|^{2/\mu}_{L^{\infty}(\Omega_T)}\leq C \bigg(\|\Lambda_{q^{(1)}} - \Lambda_{q^{(2)}}\|^{\mu/2}+\Big|\log\|\Lambda_{q^{(1)}} - \Lambda_{q^{(2)}}\|\Big|^{-1}\,\bigg), 
\end{align}
where $\mu\in(0,1)$ and the constant $C=C(\mu,\Omega,M,T)>0$.
\end{theorem}

Furthermore, we can extend the result mentioned in Theorem \ref{main result}. The stable determination of matrix potential $q$ can be derived from the measurements in a bounded subset of $\Sigma$. However, some additional information is required on the matrix potential $q$. More precisely, for $R>0$, we introduce the space $\LS_R$ as follows:
\begin{align*}
    \LS_R = \{(\overrightarrow {\phi},\overrightarrow {\psi},\overrightarrow {f})\in \LS: \overrightarrow {f}(t,x',x_3)=\overrightarrow {0},~ t\in (0,T),~x'\in\omega,~|x_3|\geq R \}.
\end{align*}
Let us also define $\Lambda^{(R)}_{{q}}$ the input-output map associated with the subset of lateral boundary
 \begin{align*}
\Lambda^{(R)}_{{q}}: \mathbf{L}_R  \quad & \rightarrow \quad \mathbf{H}^{1} (\Omega)\times\mathbf{L}^{2}(\Omega)\times \LS^2(\Sigma_R)
\end{align*}
\begin{equation}\label{input-output map_R}
\Lambda^{(R)}_{{q}}~\Big(\overrightarrow {\phi},\overrightarrow {\psi},\overrightarrow {f}\Big) := \Big(\overrightarrow {u} \big | _{t=T}, \partial_t \overrightarrow {u}\big |_{t=T}, \partial_\nu \overrightarrow {u}\big |_ {\Sigma_R}\Big),
\end{equation}
where $\overrightarrow {u}$ is the solution of \eqref{hprs} and $\Sigma_R:=(0,T)\times\partial\omega\times(-R,R)$. Consequently, the stability result can be stated as follows:
\begin{theorem}\label{main result_with_partial_data}
Let ${q^{(1)}}, {q^{(2)}}\in W^{1,\infty}(\Omega_T)$ be two sets of matrix potentials with $\|q^{(k)}\|\leq M,$ for $k=1,2$. Moreover, we assume that there exists $r>0$ for which
\begin{align}\label{q_in_bounded_domain}
    \big\|q^{(1)}-q^{(2)} \big\|_{L^{\infty}(\Omega_T)}=\big\|q^{(1)}-q^{(2)} \big\|_{L^{\infty}((0,T)\times\omega\times(-r,r))}.
\end{align}
Let $\overrightarrow {u}^{(k)}$ be the solution of \eqref{hprs} corresponding to the matrix potential $q= q^{(k)}$ and $\Lambda^{(R)}_{q^{(k)}}$ be the given input-output map defined by \eqref{input-output map_R} corresponding to $\overrightarrow {u}^{(k)}$. Then, the following stability estimate holds for all $R>r$ and $T>Diam(\omega)$,
\begin{align}\label{main_stability_estimate_for_q_with_partial_boundary_data}
    \big\|q^{(1)}-q^{(2)} \big\|^{2/\mu}_{L^{\infty}(\Omega_T)}\leq C \bigg(\|\Lambda^{(R)}_{q^{(1)}} - \Lambda^{(R)}_{q^{(2)}}\|^{\mu/2}+\Big|\log\|\Lambda^{(R)}_{q^{(1)}} - \Lambda^{(R)}_{q^{(2)}}\|\Big|^{-1}\,\bigg), 
\end{align}
where $\mu\in(0,1)$ and the constant $C=C(\mu,\Omega,M,T,R)>0$.
\end{theorem}
\begin{remark}
    Since the IBVP \eqref{hprs} is defined over the infinite waveguide $\Omega_R$, the Theorem \ref{main result_with_partial_data} can not be derived from results available for the bounded domain. However, condition \eqref{q_in_bounded_domain} holds if $q^{(1)}=q^{(2)}$ for all $(t,x)$ lies outside $(0,T)\times\omega\times (-r,r)$. Hence due to \eqref{q_in_bounded_domain}, it will be sufficient to determine the matrix potential in the bounded domain $(0,T)\times\omega\times (-r,r)$. 
\end{remark}
\section{Geometric Optics solutions}\label{sec:GO_soln}
In this section, we give the construction of exponential growing and decaying solutions which will be instrumental for our stability result. More precisely, following the ideas from \cite{kian2013stability} used for constructing geometric optics (GO) solutions for a scalar wave equation in an infinite waveguide, we construct the suitable GO solutions for the system of wave equations considered in the present article. We also remark that the GO solutions for a system of wave equations in a bounded domain are constructed in \cite{mishra2021determining}. However, the approach used in \cite{mishra2021determining} can not be carried out for our case. To overcome this, we decompose the operator $\square:=\pl_t^2 - \Delta_{x}$ as $\pl_t^2 - \Delta_{x'}$ and $-\pl^2_{x_3}$ considering $x=(x',x_3),~ x'\in\omega,x_3\in\R$. 

We prove the following Lemma in which $\mathbb{S}^1 := \{y\in\R^2: |y|=1\}$, $\mathcal{S}(\R)$ denotes the Schwarz space over $\R$ and $C_{0}^{\infty}(\R^2)$ denote the space all smooth functions having compact support in $\R^2$.
\begin{lemma}\label{gocons5}
Let $q\in W^{1,\infty}(\Omega_T)$ be a matrix-valued potential, $\theta \in  \mathbb{S}^1$, $h \in \mathcal{S}(\mathbb{R})$  and $\varphi\in \mathcal{C}_{0}^{\infty}(\mathbb{R}^{2})$ be given. Then for any $\rho > 0$, the equation
\begin{equation}\label{wave equation GO}
\partial^{2}_{t}\overrightarrow {u}(t,x) - \Delta \overrightarrow {u}(t,x) + q(t,x)\overrightarrow {u}(t,x) = \overrightarrow {0}, \qquad  (t,x)\in\Omega_T, 
\end{equation}
admits a solution of the form 
\begin{equation}\label{hypdesoln7}
\overrightarrow {u}^{\pm}(t,x) = e^{\pm i\rho(x' \cdot \theta + t)} \varphi(x'+ t\theta)h(x_{3})\overrightarrow {K}^{\pm}+ \overrightarrow {\Psi}^{\pm}(t,x\,;\rho),\ \  t\in(0,T),\,x'\in\omega,\,x_3\in\mathbb{R},
\end{equation}
where $\overrightarrow {K}^{\pm}\in \mathbb{R}^n$ is any constant vector and $\overrightarrow {\Psi}^{\pm}(t,x\,;\rho)$ satisfies
\begin{equation}\label{go_temp1} 
\begin{split}
    \overrightarrow {\Psi}^{+}(0,x\,;\rho)&= \partial_{t}\overrightarrow {\Psi}^{+}(0,x\,;\rho) = \overrightarrow {0},\qquad  x\in \Omega,\\
 \overrightarrow {\Psi}^{-}(T,x\,;\rho)& = \partial_{t}\overrightarrow {\Psi}^{-}(T,x\,;\rho) = \overrightarrow {0},\qquad x\in \Omega,\\
\overrightarrow {\Psi}^{\pm}(t,x\,;\rho)& = \overrightarrow {0},\qquad (t,x)\in\Sigma. 
\end{split}
\end{equation}
Moreover, there exists a constant $C > 0$ depending only on $\omega, T$ and $\|q\|_{W^{1,\infty}(\Omega_T)}$ such that
\begin{align}\label{psi inequ}
\rho\| \overrightarrow {\Psi}^{\pm}(\cdot;\rho)\|_{\mathbf{L}^{2}(\Omega_T)}+ \|\nabla \overrightarrow {\Psi}^{\pm}(\cdot;\rho)\|_{L^{2}(\Omega_T)}\leq C \|\overrightarrow {K}^{\pm}\|_{\mathbb{R}^n} \|\varphi\|_{H^3(\mathbb{R}^2)}\|h\|_{H^2(\mathbb{R})}.
\end{align}
\end{lemma}
\begin{proof}
We give the proof for the construction of $\overrightarrow {u}^{+}$ while that of construction of $\overrightarrow {u}^{-}$ follows similarly. 
To begin with, we first  observe the following
    $$(\partial^{2}_{t} - \Delta_{x'})\big[\varphi(x'+ t\theta)h(x_{3})e^{  i\rho(x' \cdot \theta + t)}\big]=e^{  i\rho(x' \cdot \theta + t)}(\partial^{2}_{t} - \Delta_{x'} )\big[\varphi(x'+ t\theta)h(x_{3})\big], $$ 
and similarly 
    $$-\partial^{2}_{x_3}\big[\varphi(x'+ t\theta)h(x_{3})e^{ i\rho(x' \cdot \theta + t)}\big]=e^{  i\rho(x' \cdot \theta + t)}\big[-\partial^{2}_{x_3}(\varphi(x'+ t\theta)h(x_{3}))\big].$$ 
As a consequence, we obtain
\begin{align*}
(\partial^{2}_{t} - \Delta+q(t,x))\big[e^{i\rho(x' \cdot \theta + t)}\varphi(x'+ t\theta)h(x_{3})\overrightarrow {K}^{+}\big] 
=e^{ i\rho(x' \cdot \theta + t)} \overrightarrow {J}(t,x',x_3),
\end{align*}
where $\overrightarrow {J}(t,x',x_3)$ is given by
\begin{align}\label{defn_J}
\overrightarrow {J}(t,x',x_3)=(\partial^{2}_{t} - \Delta )\big[\varphi(x'+ t\theta)h(x_{3})\overrightarrow {K}^+\big]+ q(t,x)\big[\varphi(x'+ t\theta)h(x_{3})\overrightarrow {K}^+\big].
\end{align}
Now in order to have $\overrightarrow {u}^{+}$ given by \eqref{hypdesoln7} is a solution to \eqref{wave equation GO}, it is enough to 
show the existence of $\overrightarrow {\Psi}^+$ satisfying the following equation
    \begin{equation}\label{system for Psi+}
        \begin{cases}
           \partial^{2}_{t}\overrightarrow {\Psi}^+(t,x) - \Delta \overrightarrow {\Psi}^+(t,x) + q(t,x)\overrightarrow {\Psi}^+(t,x) =-e^{ i\rho(x' \cdot \theta + t)} \overrightarrow {J}(t,x',x_3), &  (t,x)\in \Omega_T,\\
            \overrightarrow {\Psi}^{+}(0,x\,;\rho) = \partial_{t}\overrightarrow {\Psi}^{+}(0,x\,;\rho) = \overrightarrow {0}, & x\in \Omega,\\
           \overrightarrow {\Psi}^+(t,x) = \overrightarrow {0}, &  (t,x)\in \Sigma.
        \end{cases}
    \end{equation} 
From \eqref{defn_J}, we find that $e^{ i\rho(x' \cdot \theta + t)} \overrightarrow {J}(t,x',x_3)\in \mathbf{L}^2(\Omega_T)$. Following the similar analysis carried out to demonstrate the well-posedness of \eqref{equtn for v}, we ensure the existence of $\overrightarrow {\Psi}^+\in \mathbf{L}^2(0,T;\mathbf{H}_0^1(\Omega))\cap\mathbf{H}^1(0,T;\mathbf{L}^2(\Omega))$ as a solution of \eqref{system for Psi+} (also refer to the Theorem $8.1$ in \cite[Chapter 3]{lions2012non}). Moreover, $e^{ i\rho(x' \cdot \theta + t)} \overrightarrow {J}(t,x',x_3)\in \mathbf{H}^1(0,T;\mathbf{L}^2(\Omega))$ and again applying the Theorem $2.1$ in \cite[Chapter 5]{lions1968problemes}, we conclude that $\overrightarrow {\Psi}^+\in \mathbf{H}^2(\Omega_T)$.

Let us define
\begin{align}\label{antiderivative_of_Psi_is_W}
    \overrightarrow {W}(t,x)=\int_0^t\overrightarrow {\Psi}^+(s,x)ds,\qquad (t,x)\in\Omega_T,
\end{align}
and consequently, $\overrightarrow {W}$  satisfies 
\begin{align}\label{eqn_W}
    \begin{cases}
        \partial^{2}_{t}\overrightarrow {W}(t,x) - \Delta\overrightarrow {W}(t,x) + q(t,x)\overrightarrow {W}(t,x)=\overrightarrow {R}^1(t,x)+\overrightarrow {R}^2(t,x),&  ( t,x)\in\Omega_T, \\
        \overrightarrow {W} (0,x) = \partial_{t}\overrightarrow {W} (0,x) = \overrightarrow {0}, & x\in \Omega,\\
          \overrightarrow {W}(t,x) = \overrightarrow {0}, & (t,x)\in\Sigma,
    \end{cases}
\end{align}
where $\overrightarrow {R}^1$ and $\overrightarrow {R}^2$ are given by
\begin{align*}
\overrightarrow {R}^1(t,x)&=-\int_0^t e^{ i\rho(x' \cdot \theta + s)} \overrightarrow {J}(s,x',x_3)\,ds,\\
\overrightarrow {R}^2(t,x)&=\int_0^t [q(t,x)-q(s,x)]\overrightarrow {\Psi}^+(s,x)ds,
\end{align*}
respectively. Let $\tau\in [0,T]$. From the energy estimate of \eqref{eqn_W}, we get
\begin{align}\label{GO_temp2}
\|\overrightarrow {\Psi}^+(\tau,\cdot)\|^2_{\LS^2(\Omega)}&=\|\partial_t\overrightarrow {W}(\tau,\cdot)\|^2_{\LS^2(\Omega)}
        \leq C \Big(\|\overrightarrow {R}^1 +\overrightarrow {R}^2\|^2_{L^1(0,\tau;\LS^2(\Omega))}\Big)\nonumber\\
        &\leq C \Big( \|\overrightarrow {R}^1\|^2_{\LS^2(\Omega_T)}+ \|\overrightarrow {R}^2\|^2_{L^2(0,\tau;\LS^2(\Omega))} 
        \Big).
\end{align}    
    Moreover, we have
    \begin{align*}
     \overrightarrow {R}^1 (t,x) & =-\int_0^t e^{ i\rho(x' \cdot \theta + s)} \overrightarrow {J}(s,x',x_3)\,ds =- \frac{1}{i\rho}  \int_0^t\partial_s e^{ i\rho(x' \cdot \theta + s)} \overrightarrow {J}(s,x',x_3)\,ds\\
        &=\frac{1}{i\rho}  \int_0^t e^{ i\rho(x' \cdot \theta + s)} \partial_s\overrightarrow {J}(s,x',x_3)\,ds
        -\frac{e^{ i\rho(x' \cdot \theta + t)} \overrightarrow {J}(t,x',x_3)-e^{ i\rho(x' \cdot \theta )} \overrightarrow {J}(0,x',x_3)}{i\rho}, 
    \end{align*}
   and it follows from \eqref{defn_J}
 \begin{align}\label{go_temp3}
     \|\overrightarrow {R}^1\|_{\LS^2(\Omega_T)} \leq \frac{C}{\rho} 
     \|\overrightarrow {K}^+\|_{\mathbb{R}^n} \|h\|_{H^2(\mathbb{R})}\|\varphi\|_{H^3(\mathbb{R}^2)}.
 \end{align}
Again estimate for $\overrightarrow {R}^2$ can  be obtained by
\begin{align}\label{go_temp3}
    \|\overrightarrow {R}^2\|^2_{L^2(0,\tau;\LS^2(\Omega))} \leq C \|q\|^2_{W^{1,\infty}(\Omega_T)} 
    \int_0^\tau \|\pl_t \overrightarrow {W}(s,\cdot)\|^2_{\LS^2(\Omega)}~ds.
\end{align}

 Hence the estimates in \eqref{GO_temp2} reduces to
\begin{align*}
 \|\partial_t\overrightarrow {W}(\tau,\cdot)\|^2_{\mathbf{L}^2(\Omega)} &\leq C \Big(\frac{1}{\rho^2} 
     \|\overrightarrow {K}^+\|^2_{\mathbb{R}^n} \|h\|^2_{H^2(\mathbb{R})}\|\varphi\|^2_{H^3(\mathbb{R}^2)}+ \|q\|^2_{W^{1,\infty}(\Omega_T)} 
    \int_0^\tau \|\pl_t \overrightarrow {W}(s,\cdot)\|^2_{\LS^2(\Omega)}~ds \Big).
\end{align*}
 Applying the Gronwall's inequality, we deduce 
 \begin{align*}
     \|\partial_t\overrightarrow {W}(\tau,\cdot)\|^2_{\mathbf{L}^2(\Omega)} \leq &\frac{C}{\rho^2} 
     \|\overrightarrow {K}^+\|^2_{\mathbb{R}^n} \|h\|^2_{H^2(\mathbb{R})}\|\varphi\|^2_{H^3(\mathbb{R}^2)},
 \end{align*}
 for almost everywhere $\tau\in [0,T].$
 Incorporating \eqref{antiderivative_of_Psi_is_W} we have
 \begin{align}\label{energy_est_for_Psi}
     \|\overrightarrow {\Psi}^+\|_{\mathbf{L}^2(\Omega_T)}\leq  \frac{C}{\rho} 
     \|\overrightarrow {K}^+\|_{\mathbb{R}^n} \|h\|_{H^2(\mathbb{R})}\|\varphi\|_{H^3(\mathbb{R}^2)}.
 \end{align} 
 Combining the inequality \eqref{energy_est_for_Psi} along with the coercivity of the bilinear form
 \begin{align*}
     a(t; \overrightarrow {g},\overrightarrow {h}):= \int_{\Omega} \Big(\nabla \overrightarrow {g}(x)\cdot\overline{\nabla\overrightarrow {h}(x)}+ q(t,x)\overrightarrow {g}(x)\cdot\overline{\overrightarrow {h}(x)} \Big)~dx
 \end{align*}
associated to \eqref{system for Psi+} and applying \cite{kian2017unique} and \cite[Theorem 8.1]{lions2012non}, we deduce 
\begin{align*}
    \|\nabla \overrightarrow {\Psi}^{+}(\cdot;\rho)\|_{L^{2}(\Omega_T)}\leq C \|\overrightarrow {K}^{+}\|_{\mathbb{R}^n} \|\varphi\|_{H^3(\mathbb{R}^2)}\|h\|_{H^2(\mathbb{R})}
\end{align*}
and hence the estimate \eqref{psi inequ} is obtained. In an analogous way, the existence of $\overrightarrow {u}^-$ can be ensured and consequently, using \eqref{go_temp1}, the estimate \eqref{psi inequ} can be obtained. Hence the result follows.
\end{proof}
We make the following remark which will be required to study the stability estimate:
\begin{remark}
Using the fact that $\overrightarrow {\Psi}^{\pm}$ is vanishing at the boundary of $\Omega$, we have 
\begin{equation*}
    \begin{split}
        \overrightarrow {u}^{\pm}(t,x)=\overrightarrow {K}^{\pm}\left(\varphi(x'+ t\theta)h(x_{3})e^{ i\pm\rho(x' \cdot \theta + t)}\right) &\quad \text{on}\,\, \Sigma.
    \end{split}
\end{equation*}
Consequently, using the Theorem 2.2 in \cite[Chapter 4]{lions1968problemes}, we have
\begin{equation}\label{u1_at_bdry}
    \|\overrightarrow {u}^{\pm}\|_{\LS}\leq C \rho^2\|\overrightarrow {K}^{\pm}\|^2_{\mathbb{R}^n}\|\varphi\|_{H^{3}(\mathbb{R}^{2})}\|h\|_{H^{2}(\mathbb{R})}. 
\end{equation}
\end{remark}

\section{stability estimate}\label{sec:stab_esti}
We devote this section to proving our main stability result. 
In order to establish the stability estimate in Theorem \ref{main result}, we need to derive an appropriate integral identity.
 \begin{proposition}\label{propintid}
 {Let $\displaystyle q^{(k)}:=((q^{(k)}_{ij}))_{1\leq i,j\leq n}$ for $k=1,2$  be two sets of matrix potentials with  with $q^{(1)}_{ij},q^{(2)}_{ij}\in W^{1,\infty}(\Omega_T)$ for all $1\leq i,j\leq n$ and $\|q^{(k)}\|\leq M,$ for $k=1,2$.
Let $\overrightarrow {u}^{(k)}$ be the solution of \eqref{hprs} corresponding to the matrix potential $q= q^{(k)}$ and $\Lambda_{q^{(k)}}$ be the given IO map defined by \eqref{input-output map} corresponding to $\overrightarrow {u}^{(k)}$.}
Further, assume that $\Lambda_{q^{(1)}} \neq \Lambda_{q^{(2)}}$, then we have
\begin{equation}\label{intid2}
\begin{split}
    \int_{\Omega_T}&(q(t,x) \overrightarrow {u}^{(1)})\cdot\overrightarrow {v}~dx~dt=
    - \int_{\Sigma} \Big(\Lambda^{3}_{q^{(2)}}-\Lambda^{3}_{q^{(1)}}\Big)(\overrightarrow {\phi},\overrightarrow {\psi},\overrightarrow {f}) \cdot\overrightarrow {v}~dS~dt\\
    & + \int_{\Omega}\Big\{\Big(\Lambda^{2}_{q^{(2)}}-\Lambda^{2}_{q^{(1)}}\Big)(\overrightarrow {\phi},\overrightarrow {\psi},\overrightarrow {f})\cdot\overrightarrow {v}(T,\cdot)
     -\partial_{t}\overrightarrow {v}(T,\cdot)\cdot\Big(\Lambda^{1}_{q^{(2)}}-\Lambda^{1}_{q^{(1)}}\Big)(\overrightarrow {\phi},\overrightarrow {\psi},\overrightarrow {f})\Big\}dx,
        \end{split}
    \end{equation}
where $\overrightarrow {v}$ satisfies the adjoint equation $\mathcal{L}_{q^{(2)}}^{*}\overrightarrow {v}(t,x) = \overrightarrow {0}$ in $\Omega_T$.
\end{proposition}
\begin{proof}
We define $\overrightarrow {u} = \overrightarrow {u}^{(2)} - \overrightarrow {u}^{(1)}$ and consequently, $\overrightarrow {u}$ satisfies the following IBVP:
\begin{equation}\begin{cases}\label{u-solves}
    \mathcal{L}_{q^{(2)}}\overrightarrow {u} = q(t,x) \overrightarrow {u}^{(1)}, &~\text{in}~ \Omega_T\\
    \overrightarrow {u}(0,\cdot)= \overrightarrow {0}, \, \partial_t \overrightarrow {u}(0,\cdot) = \overrightarrow {0}, &~\text{in}~\Omega\\
    \overrightarrow {u} = \overrightarrow {0}, &~\text{on}~  \Sigma.
\end{cases}
\end{equation} 
With the help of integration by parts, we find that
\begin{align}\label{inteqn}
    \int_{\Omega_{T}} & {\mathcal{L}_{q^{(2)}}\overrightarrow {u}(t,x) \cdot \overrightarrow {v}(t,x)}\,dt\,dx  = \int_{\Omega}{\left[\partial_{t}\overrightarrow {u}(T,x)\cdot \overrightarrow {v}(T,x) - \partial_{t}\overrightarrow {u}(0,x)\cdot \overrightarrow {v}(0,x)\right]}\,dx \nonumber\\
&  - \int_{\Omega}{[\overrightarrow {u}(T,x)\cdot\partial_{t}\overrightarrow {v}(T,x) - \overrightarrow {u}(0,x)\cdot \partial_{t}\overrightarrow {v}(0,x)]}\,dx
 + \int_{\Omega_{T}}{\overrightarrow {u}(t,x ) \cdot \mathcal{L}_{q^{(2)}}^{*}\overrightarrow {v}(t,x)}\,dt\,dx \nonumber\\
    &  -\int_{\Sigma}{[\partial_{\nu}\overrightarrow {u}(t,x)\cdot \overrightarrow {v}(t,x) - \overrightarrow {u}\cdot \partial_{\nu}\overrightarrow {v}(t,x)]}\,dt\,dS_{x}.
\end{align}
Since $\overrightarrow {u}(T,x)= \overrightarrow {u}^{(2)}(T,x)-\overrightarrow {u}^{(1)}(T,x)=\Big(\Lambda^{1}_{q^{(2)}}-\Lambda^{1}_{q^{(1)}}\Big)(\overrightarrow {\phi},\overrightarrow {\psi},\overrightarrow {f})$, and in a similar manner, incorporating the other components of $\Lambda_{q^{(k)}}$, using \eqref{u-solves}, the equation \eqref{inteqn} provides the integral identity \eqref{intid2}.
\end{proof}
The following result follows from the Proposition \ref{intid2}.
\begin{corollary}\label{equal io map}
{Let $\displaystyle q^{(k)}:=((q^{(k)}_{ij}))_{1\leq i,j\leq n}$ for $k=1,2$  be two sets of matrix potentials with  with $q^{(1)}_{ij},q^{(2)}_{ij}\in W^{1,\infty}(\Omega_T)$ for all $1\leq i,j\leq n$ and $\|q^{(k)}\|\leq M,$ for $k=1,2$.
Let $\overrightarrow {u}^{(k)}$ be the solution of \eqref{hprs} corresponding to the matrix potential $q= q^{(k)}$ and $\Lambda_{q^{(k)}}$ be the given IO map defined by \eqref{input-output map} corresponding to $\overrightarrow {u}^{(k)}$.}
 Further, assume that $\Lambda_{q^{(1)}} = \Lambda_{q^{(2)}}$, then the following integral identity holds
\begin{equation}\label{intid1}
    \int_{\Omega_{T}}{ \left(q(t,x)\overrightarrow {u}^{(1)}(t,x)\right)\cdot \overrightarrow {v}(t,x)\,dt\,dx } = 0,
\end{equation}
where $\overrightarrow {v}$ satisfies the adjoint equation $\mathcal{L}_{q^{(2)}}^{*}\overrightarrow {v}(t,x) = \overrightarrow {0}$ in $\Omega_T$.
\end{corollary}

To prove the Theorem \ref{main result}, we need the following lemma in which using the GO solutions, an estimate on the difference of matrix potentials $q^{(1)}$ and $q^{(2)}$  is derived.
\begin{lemma}\label{lemma 1 for stab}
Let $q^{(1)},q^{(2)}\in W^{1,\infty}(\Omega_T)$, with $\|q^{(j)}\|\leq M, j=1,2,$ and let matrix potential $q$ of size $n\times n$ be equal to $q^{(1)}-q^{(2)}$ and extended by zero outside of $\Omega_T$ . Then, for all $\overrightarrow {K}, \overrightarrow {K}^{(*)}\in\mathbb{R}^n$, $\theta \in  \mathbb{S}^1$, $h \in \mathcal{S}(\mathbb{R})$  and $\varphi\in \mathcal{C}_{0}^{\infty}(\mathbb{R}^{2})$ we have 
\begin{equation}\label{first ineq for stab}
  \begin{split}
\biggr|\int_{\mathbb{R}}&\int_{\mathbb{R}^2}\int_{0}^{T}\Big(q(t,x',x_3)\overrightarrow {K}\cdot\overrightarrow {K}^{(*)}\Big)\varphi^{2}(x'+ t\theta)h^2(x_3)\,dt\,dx'\,dx_3\biggr|\\ 
& \leq C\bigg(\frac{1}{\rho} + \rho^4 \|\Lambda_{q^{(2)}} - \Lambda_{q^{(1)}}\| \bigg)\|\varphi\|^2_{H^{3}(\mathbb{R}^{2})}\|h\|^2_{H^{2}(\mathbb{R})},
  \end{split}
\end{equation}
for any $\rho >1$ sufficiently large.  Here the constant $C>0$ depends on $\Omega, T, M, \overrightarrow {K}$ and $\overrightarrow {K}^{(*)}$.
\end{lemma}
\begin{proof}
With reference to the Lemma \ref{gocons5}, for each $j=1,2$, we can choose the GO solutions of
$$\partial^{2}_{t}\overrightarrow {u}^{(j)}_{\rho}(t,x) - \Delta \overrightarrow {u}^{(j)}_{\rho}(t,x) + q^{(j)}(t,x)\overrightarrow {u}^{(j)}_{\rho}(t,x) = \overrightarrow {0},\qquad (t,x)\in\Omega_T,$$
in the following form:
\begin{align}\label{GO_sol_u_i}
   \overrightarrow {u}^{(j)}_{\rho}(t,x',x_3;\rho) = \overrightarrow {K}\,\varphi(x'+ t\theta)h(x_{3})e^{ i\rho(x' \cdot \theta + t)} +\overrightarrow {\Psi}^{(j)}(t,x\,;\rho),
\end{align}
where $\overrightarrow {\Psi}^{(j)},~j=1,2$ satisfies
\begin{equation*}
\begin{split}
    \overrightarrow {\Psi}^{(j)}(0,x\,;\rho)&= \partial_{t}\overrightarrow {\Psi}^{(j)}(0,x\,;\rho) = \overrightarrow {0},\qquad  x\in \Omega,\\
\overrightarrow {\Psi}^{(j)}(t,x\,;\rho)& = \overrightarrow {0},\qquad (t,x)\in\Sigma. 
\end{split}
\end{equation*}
Let us define $\overrightarrow {u}_{\rho}=\overrightarrow {u}^{(2)}_{\rho}-\overrightarrow {u}^{(1)}_{\rho}$. We observe that $\overrightarrow {u}_{\rho}$ satisfies the IBVP \eqref{u-solves}. Let $(q^{(2)})^*$ denote the adjoint of the matrix potential $q^{(2)}$. Furthermore, we choose GO solution $\overrightarrow {v}$ of
$$\partial^{2}_{t}\overrightarrow {v}(t,x) - \Delta \overrightarrow {v}(t,x) + (q^{(2)})^*(t,x)\overrightarrow {v}(t,x) = \overrightarrow {0},\qquad (t,x)\in\Omega_T,$$ 
in the following form (cf. Lemma \ref{gocons5})
\begin{align}\label{adjoint_v_go_sol}
   \overrightarrow {v}(t,x',x_3;\rho) = \overrightarrow {K}^{(*)}\,\varphi(x'+ t\theta)h(x_{3})e^{-i\rho(x' \cdot \theta + t)} +\overrightarrow {\Psi}^{(*)}(t,x\,;\rho)
\end{align}
in which $\overrightarrow {\Psi}^{(*)}$ satisfies
\begin{equation*}
\begin{split}
    \overrightarrow {\Psi}^{(*)}(T,x\,;\rho)&= \partial_{t}\overrightarrow {\Psi}^{(*)}(T,x\,;\rho) = \overrightarrow {0},\qquad  x\in \Omega,\\
\overrightarrow {\Psi}^{(*)}(t,x\,;\rho)& = \overrightarrow {0},\qquad (t,x)\in\Sigma. 
\end{split}
\end{equation*}
 Substituting the GO solutions \eqref{GO_sol_u_i}and \eqref{adjoint_v_go_sol} into the integral identity \eqref{intid2}, we get
\begin{equation}\label{integral1}
\begin{split}
\int_{\mathbb{R}}&\int_{\mathbb{R}^2}\int_{0}^{T}\Big(q(t,x',x_3)\overrightarrow {K}\cdot\overrightarrow {K}^{(*)}\Big)\varphi^{2}(x'+ t\theta)h^2(x_3)\,dt\,dx'\,dx_3\\ =&\int_{\Omega_T}R_{\rho}\,dt\,dx- \int_{\Sigma} \big(\Lambda^{3}_{q^{(2)}}-\Lambda^{3}_{q^{(1)}}\Big)(\overrightarrow {\phi}_{\rho},\overrightarrow {\psi}_{\rho},\overrightarrow {f}_{\rho}) \cdot\overrightarrow {v}~dt~dS\\&+\int_{\Omega}\Big\{\Big(\Lambda^{2}_{q^{(2)}}-\Lambda^{2}_{q^{(1)}}\Big)(\overrightarrow {\phi}_{\rho},\overrightarrow {\psi}_{\rho},\overrightarrow {f}_{\rho})\cdot\overrightarrow {v}(T,\cdot)-\partial_{t}\overrightarrow {v}(T,\cdot)\cdot\Big(\Lambda^{1}_{q^{(2)}}-\Lambda^{1}_{q^{(1)}}\Big)(\overrightarrow {\phi}_{\rho},\overrightarrow {\psi}_{\rho},\overrightarrow {f}_{\rho})\Big\}dx,
\end{split}
\end{equation}
where $R_{\rho}$ is given by
\begin{align*}
   R_\rho(t,x',x_3)&=-q(t,x',x_3)\Big(~\overrightarrow {K}^{(*)}\cdot\overrightarrow {\Psi}^{(1)}(t,x\,;\rho)\varphi(x'+t\theta)h(x_{3})e^{- i\rho(x' \cdot \theta + t)}\\
        &+\overrightarrow {K}\cdot\overrightarrow {\Psi}^{(*)}(t,x\,;\rho)\varphi(x'+ t\theta)h(x_{3})e^{i\rho(x' \cdot \theta + t)}+\overrightarrow {\Psi}^{(1)}(t,x',x_3)\cdot\overrightarrow {\Psi}^{(*)}(t,x',x_3)\Big).
\end{align*}
Using the H\"{o}lder's inequality and the estimate \eqref{psi inequ}, we have
  \begin{align}\label{Rrho}
      \int_{\Omega_T} |R_\rho| \,dx\,dt \leq \frac{C}{\rho} \max\big(1,\lVert \overrightarrow {K}\rVert^2_{\mathbb{R}^n}, \lVert \overrightarrow {K}^{(*)} \rVert^2_{\mathbb{R}^n}\big)\|q\|_{L^{\infty}(\Omega_T)}\|\varphi\|^2_{H^3(\mathbb{R}^2)}\|h\|^2_{H^2(\mathbb{R})},
      \end{align}
  where the constant $C>0$ depends on $\omega,T$. 
  Moreover, again using the H\"{o}lder's inequality and the Minkowski's inequality, we have 
  \begin{align}\label{estimates_for_input_component}
    \int_{\Omega}&\Big\{\Big(\Lambda^{2}_{q^{(2)}}-\Lambda^{2}_{q^{(1)}}\Big)(\overrightarrow {\phi}_{\rho},\overrightarrow {\psi}_{\rho},\overrightarrow {f}_{\rho})\cdot\overrightarrow {v}(T,\cdot)- \partial_{t}\overrightarrow {v}(T,\cdot)\cdot\Big(\Lambda^{1}_{q^{(2)}}-\Lambda^{1}_{q^{(1)}}\Big)(\overrightarrow {\phi}_{\rho},\overrightarrow {\psi}_{\rho},\overrightarrow {f}_{\rho})\Big\}dx\nonumber\\
    & \qquad- \int_{\Sigma} \Big(\Lambda^{3}_{q^{(2)}}-\Lambda^{3}_{q^{(1)}}\Big)(\overrightarrow {\phi}_{\rho},\overrightarrow {\psi}_{\rho},\overrightarrow {f}_{\rho}) \cdot\overrightarrow {v}~dt~dS\nonumber\\
    & \leq \|\Big(\Lambda^{2}_{q^{(2)}}-\Lambda^{2}_{q^{(1)}}\Big)(\overrightarrow {\phi}_{\rho},\overrightarrow {\psi}_{\rho},\overrightarrow {f}_{\rho})\|_{L^2(\Omega)}\|\overrightarrow {v}(T,\cdot)\|_{\LS^2(\Omega)}\nonumber\\
    & \qquad +\|\Big(\Lambda^{3}_{q^{(2)}}-\Lambda^{3}_{q^{(1)}}\Big)(\overrightarrow {\phi}_{\rho},\overrightarrow {\psi}_{\rho},\overrightarrow {f}_{\rho})\|_{L^2(\Sigma)}\|\overrightarrow {v}\|_{\LS^2(\Sigma)}\nonumber\\
    & \qquad+\|\Big(\Lambda^{1}_{q^{(2)}}-\Lambda^{1}_{q^{(1)}}\Big)(\overrightarrow {\phi}_{\rho},\overrightarrow {\psi}_{\rho},\overrightarrow {f}_{\rho})\|_{L^2(\Omega)}\|\partial_t\overrightarrow {v}(T,\cdot)\|_{\LS^2(\Omega)}\nonumber\\
     &\leq\Big(\|\Big(\Lambda^{1}_{q^{(2)}}-\Lambda^{1}_{q^{(1)}}\Big)(\overrightarrow {\phi}_{\rho},\overrightarrow {\psi}_{\rho},\overrightarrow {f}_{\rho})\|^2_{H^1(\Omega)}+\|\Big(\Lambda^{2}_{q^{(2)}}-\Lambda^{2}_{q^{(1)}}\Big)(\overrightarrow {\phi}_{\rho},\overrightarrow {\psi}_{\rho},\overrightarrow {f}_{\rho})\|^2_{L^2(\Omega)}\nonumber\\
     &~+\|\Big(\Lambda^{3}_{q^{(2)}}-\Lambda^{3}_{q^{(1)}}\Big)(\overrightarrow {\phi}_{\rho},\overrightarrow {\psi}_{\rho},\overrightarrow {f}_{\rho})\|^2_{L^2(\Sigma)}\Big)^{\frac{1}{2}} \Big(\|\overrightarrow {v}(T,\cdot)\|^2_{\mathbf{L}^2(\Omega)}+\|\partial_t\overrightarrow {v}(T,\cdot)\|^2_{\mathbf{L}^2(\Omega)}+\|\overrightarrow {v}\|^2_{\mathbf{L}^2(\Sigma)}\Big)^{\frac{1}{2}} \nonumber\\
    &=\|\Big(\Lambda_{q^{(2)}}-\Lambda_{q^{(1)}}\Big)(\overrightarrow {\phi}_{\rho},\overrightarrow {\psi}_{\rho},\overrightarrow {f}_{\rho})\|_{\mathbf{H}^1(\Omega)\times \mathbf{L}^2(\Omega)\times \mathbf{L}^2(\Sigma)} \|\overrightarrow {g}_{\rho}\|_{\mathbf{L}^2(\Omega)\times \mathbf{L}^2(\Omega)\times \mathbf{L}^2(\Sigma)},
\end{align}
  where $(\overrightarrow {\phi}_{\rho},\overrightarrow {\psi}_{\rho},\overrightarrow {f}_{\rho})= \Big(\overrightarrow {u}^{(1)}_{\rho}\Big|_{t=0},\partial_t\overrightarrow {u}^{(1)}_{\rho}\Big|_{t=0},\overrightarrow {u}^{(1)}_{\rho}\Big|_{\Sigma}\Big)$ and $\overrightarrow {g}_{\rho}$ is given by 
  $$\overrightarrow {g}_{\rho}=\Big(\overrightarrow {v}(T,\cdot),\partial_t\overrightarrow {v}(T,\cdot),\overrightarrow {v}|_{\Sigma}\Big).$$ 
  Since we assume $\rho>1$ and by the prescribed data $(\overrightarrow {\phi}_{\rho},\overrightarrow {\psi}_{\rho},\overrightarrow {f}_{\rho})$, we deduce
  \begin{align}\label{I esti}
    \|\Big(\Lambda_{q^{(2)}}&-\Lambda_{q^{(1)}}\Big)(\overrightarrow {\phi}_{\rho},\overrightarrow {\psi}_{\rho},\overrightarrow {f}_{\rho})\|_{\mathbf{H}^1(\Omega)\times \mathbf{L}^2(\Omega)\times \mathbf{L}^2(\Sigma)} \nonumber\\
    &\leq\|\Lambda_{q^{(2)}}-\Lambda_{q^{(1)}}\|\,\,\|(\overrightarrow {\phi}_{\rho},\overrightarrow {\psi}_{\rho},\overrightarrow {f}_{\rho})\|_{\mathbf{H}^1(\Omega)\times \mathbf{L}^2(\Omega)\times \mathbf{L}^2(\Sigma)},\nonumber\\ 
    &\leq\|\Lambda_{q^{(2)}}-\Lambda_{q^{(1)}}\|\,\,\Big(\|\overrightarrow {u}^{(1)}_{\rho}(0,\cdot)\|_{\mathbf{H}^1(\Omega)}+\|\partial_t\overrightarrow {u}^{(1)}_{\rho}(0,\cdot)\|_{\mathbf{L}^2(\Omega)}+\|\overrightarrow {u}^{(1)}_{\rho}|_{\Sigma}\|_{\mathbf{L}^2(\Sigma)}\Big)\nonumber\\
    &\leq C \rho^2\|\Lambda_{q^{(2)}}-\Lambda_{q^{(1)}}\|\,\Big(\|\varphi\|_{H^{3}(\mathbb{R}^{2})}\|h\|_{H^{2}(\mathbb{R})}\Big),
\end{align}
  where we have used 
   \eqref{u1_at_bdry}, \eqref{GO_sol_u_i}and \eqref{adjoint_v_go_sol} respectively. 
Moreover, using \eqref{adjoint_v_go_sol} and \eqref{u1_at_bdry} we get
  \begin{equation}\label{v at T and bdry}
      \begin{split}
          \|\overrightarrow {g}_{\rho}\|_{\mathbf{H}^1(\Omega)\times \mathbf{L}^2(\Omega)\times \mathbf{L}^2(\Sigma)}&\leq  \,\,\Big(\|\overrightarrow {v}_{|t=T}\|_{\mathbf{L}^2(\Omega)}+\|\partial_t\overrightarrow {v}_{|t=T}\|_{\mathbf{L}^2(\Omega)}+\|\overrightarrow {v}_{|\Sigma}\|_{\mathbf{L}^2(\Sigma)}\Big)\\
          &\leq C\, \rho^2\,\, \|\varphi\|_{H^{3}(\mathbb{R}^{2})}\|h\|_{H^{2}(\mathbb{R})}.
      \end{split}
  \end{equation}
By substituting the estimates obtained in \eqref{I esti} and \eqref{v at T and bdry} into \eqref{estimates_for_input_component}, it becomes
  \begin{equation}\label{IIq estim}
      \begin{split}
        \int_{\Omega}&\Big\{\Big(\Lambda^{2}_{q^{(2)}}-\Lambda^{2}_{q^{(1)}}\Big)(\overrightarrow {\phi}_{\rho},\overrightarrow {\psi}_{\rho},\overrightarrow {f}_{\rho})\cdot\overrightarrow {v}(T,\cdot)-\partial_{t}\overrightarrow {v}(T,\cdot)\cdot\Big(\Lambda^{1}_{q^{(2)}}-\Lambda^{1}_{q^{(1)}}\Big)(\overrightarrow {\phi}_{\rho},\overrightarrow {\psi}_{\rho},\overrightarrow {f}_{\rho})\Big\}dx\\ &\qquad - \int_{\Sigma} \big(\Lambda^{3}_{q^{(2)}}-\Lambda^{3}_{q^{(1)}}\Big)(\overrightarrow {\phi}_{\rho},\overrightarrow {\psi}_{\rho},\overrightarrow {f}_{\rho}) \cdot\overrightarrow {v}~dt~dS
        \\&\leq C \rho^4\|\Lambda_{q^{(2)}}-\Lambda_{q^{(1)}}\|\,\Big(\|\varphi\|^2_{H^{3}(\mathbb{R}^{2})}\|h\|^2_{H^{2}(\mathbb{R})}\Big).
      \end{split}
  \end{equation}
Combining the estimates \eqref{Rrho} and \eqref{IIq estim} together in \eqref{integral1}, we have the required estimate \eqref{first ineq for stab}. This completes the proof.
\end{proof}

\begin{remark}
To obtain the stability estimate of the matrix potential, we look to estimate $\|q_{ij}\|_{L^\infty(\Omega_T)}$, $i,j=1,2,\dots,n$. Hereby we need to choose the constant vectors $\overrightarrow {K}$ and $\overrightarrow {K}^{(*)}$ appropriately.
For $i,j\in\{1,2,\dots,n\}$, let us consider $\overrightarrow {K}, \overrightarrow {K}^{(*)}\in \mathbb{R}^n$ as follows:
\begin{equation*}
    \overrightarrow {K}=\left( \delta_{ir} \right)^n_{r=1},\,\,\quad \overrightarrow {K}^{(*)}=\left( \delta_{jr}\right)^n_{r=1},
\end{equation*}
where $\delta_{kr}$ represents the Kronecker delta function for non-negative integers $k$ and $r$.
As a consequence, we obtain
\begin{equation}\label{qij}
    q\overrightarrow {K}\cdot\overrightarrow {K}^{(*)}= \,q_{ij}.  
\end{equation}
\end{remark}
\subsection{Light-ray transform}
Afterwards, we introduce the light ray transform \cite{Stefanov LRT} of $(t,x')\mapsto q_{ij}(\cdot,\cdot,x_3)$ by fixing $x_3\in\R$.
Let $g\in L^1(\R^{4})$ be arbitrary. More precisely, we define the light ray transform $\mathcal{L}$ of a function $g(\cdot,\cdot,x_3)$ as
$$\mathcal{L}[g(\cdot,\cdot,x_3)](\theta,x'):=\int_{\R}g(t,\,x'-t\theta,\,x_3)\,dt,\quad x'\in\mathbb{R}^2,\,\theta\in\mathbb{S}^1.$$
 Suppose that 
    \begin{align}\label{extending_q_outside_of_the_domain}
        q(t,x):= {q^{(1)}}(t,x) - {q^{(2)}}(t,x), \quad (t,x)\in \Omega_T,
    \end{align}
    and it is extended by zero outside the domain ${\Omega_T}$.
Hence for $i,\,j\in \{1,2,\dots,n\}$, we have 
$$\mathcal{L}[q_{ij}(\cdot,\cdot,x_3)](\theta,x')=\int_{\mathbb{R}}q_{ij}\,(t,\,x'-t\theta,\,x_3)\,dt.$$
\begin{lemma}\label{lem:esti_q}
   Let $q$ be as in Lemma \ref{lemma 1 for stab} and $y_3\in\R$ be fixed. Then, there exist $C>0,\,\beta>0, \delta >0, \text{and}\, \rho_0>0$ such that for all $\theta\in  \mathbb{S}^1,$ we have 
   \begin{equation}\label{estimate on ray R}
       \Big| \mathcal{L}[q_{ij}(\cdot,\cdot,y_3)](\theta,y')\Big|\leq C \left(\rho^{\beta}\|\Lambda_{q^{(2)}} - \Lambda_{q^{(1)}}\| + \frac{1}{\rho^{\delta}}\right),\qquad \text{a.e.}~ y'\in\mathbb{R}^{2},
   \end{equation}
   for any $\rho\geq\rho_0$ and the constant $C$ depends only on $\Omega,\,T$ and $M$.
\end{lemma}
\begin{proof}
  Let $\varphi\in \mathcal{C}_{0}^{\infty}(\mathbb{R}^{2})$ and $h\in \mathcal{C}_{0}^{\infty}(\mathbb{R})$ be positive functions which are supported in the unit ball $B(0, 1)$ and $[-1,1]$ respectively. In addition, we consider \,$\|\varphi\|_{L^2(\mathbb{R}^2)}=1$ and $\|h\|_{L^2(\mathbb{R})}=1$. Let us define
  \begin{equation}
  \begin{split}
      \varphi_{\varepsilon}(x'):&=\varepsilon^{-1}\varphi\Big(\frac{x'-y'}{\varepsilon}\Big),\quad x',y'\in\R^2, \\
      h_{\varepsilon}(x_3):&=\varepsilon^{-1/2}h\Big(\frac{x_3-y_3}{\varepsilon}\Big),\quad x_3,y_3\in\R.
  \end{split}\end{equation}
In view of Lemma \ref{lemma 1 for stab}, replacing $\varphi$ and $h$ by $\varphi_{\varepsilon}$ and $h_{\varepsilon}$ respectively, we have
  \begin{equation*}
      \begin{split}
          &\Bigg|\int_0^Tq_{ij}(t,y'-t\theta,y_3)dt\Bigg|=\Bigg|\int_0^T\int_{\mathbb{R}}\int_{\mathbb{R}^2}q_{ij}(t,y'-t\theta,y_3)\varphi^2_{\varepsilon}(x')h^2_{\varepsilon}(x_3)dx'dx_3dt\Bigg|\\
      &\leq \Bigg|\int_0^T\int_{\mathbb{R}}\int_{\mathbb{R}^2}q_{ij}(t,x'-t\theta,x_3)\varphi^2_{\varepsilon}(x')h^2_{\varepsilon}(x_3)dx'dx_3dt\Bigg|\\
      &\qquad+\Bigg|\int_0^T\int_{\mathbb{R}}\int_{\mathbb{R}^2}(q_{ij}(t,y'-t\theta,y_3)-q_{ij}(t,x'-t\theta,x_3))\varphi^2_{\varepsilon}(x')h^2_{\varepsilon}(x_3)dx'dx_3dt\Bigg|.
      \end{split}
  \end{equation*}
Since $q\in W^{1,\infty}([0,T]\times\R^3)$, we have $$|q_{ij}(t,y'-t\theta,y_3)-q_{ij}(t,x'-t\theta,x_3)|\leq C|(y',y_3)-(x',x_3)|.$$ 
Consequently, we obtain
\begin{equation}\label{estimate on ray transform}
    \begin{split}
        \Bigg|\int_0^Tq_{ij}(t,y'-t\theta,y_3)dt\Bigg|\leq C &\Big(\rho^{4}\|\Lambda_{q^{(2)}} -  \Lambda_{q^{(1)}}\| +\frac{1}{\rho}\Big)\|\varphi_\varepsilon\|^2_{H^{3}(\mathbb{R}^{2})}\|h_\varepsilon\|^2_{H^{2}(\mathbb{R})}\\
        & +C\int_{\mathbb{R}}\int_{\mathbb{R}^2}|(y',y_3)-(x',x_3)|\varphi^2_{\varepsilon}(x')h^2_{\varepsilon}(x_3)dx'dx_3.
    \end{split}
\end{equation}
It is straightforward to observe that
\begin{align*}
    &\|\varphi_\varepsilon\|_{H^{3}(\mathbb{R}^{2})}\leq C \varepsilon^{-3},\qquad \|h_\varepsilon\|_{H^{2}(\mathbb{R})}\leq C \varepsilon^{-2},\\
    &\int_{\mathbb{R}}\int_{\mathbb{R}^2}|(y',y_3)-(x',x_3)|\varphi^2_{\varepsilon}(x')h^2_{\varepsilon}(x_3)\,dx'\,dx_3\, \leq C\,\varepsilon .
\end{align*}
As a result, \eqref{estimate on ray transform} reduces to
\begin{equation*}
    \Big|\int_0^Tq_{ij}(t,y'-t\theta,y_3)dt\Big|\leq C \left(\rho^{4}\|\Lambda_{q^{(2)}} -  \Lambda_{q^{(1)}}\| +\frac{1}{\rho}\right)\varepsilon^{-10}+ C\varepsilon.
\end{equation*}
By choosing $\varepsilon$ such a way $\varepsilon=\frac{\varepsilon^{-10}}{\rho}$, there exist constants $\delta>0$ and $\beta>0$ such that
\begin{equation*}
\Big|\mathcal{L}[q_{ij}(\cdot,\cdot,y_3)](\theta,y')\Big|=    \Bigg|\int_{\R}q_{ij}(t,y'-t\theta,y_3)dt\Bigg|\leq C \left(\rho^{\beta}\|\Lambda_{q^{(2)}} - \Lambda_{q^{(1)}}\| + \frac{1}{\rho^{\delta}}\right),~ \text{a.e.}~ y'\in\mathbb{R}^{2}.
\end{equation*}
Hence the result follows.
\end{proof}
Afterwards, we would like to define the Fourier transform acting on the component of the matrix potential and subsequently, our goal is to find its estimate. In this regard, let us fix $y_3\in \mathbb{R}$. We set 
$$E=\Big\{(\tau,\xi)\in\mathbb{R}\times (\mathbb{R}^2\backslash\{(0,0)\}):|\tau|\leq|\xi|\Big\}$$ 
and define the Fourier transform of $q_{ij}(\cdot,\cdot,y_3)\in L^1(\mathbb{R}^{3})$ as
\begin{equation*}
    \widehat{q_{ij}}[(\cdot,\cdot,y_3)](\tau,\xi):=\int_{\mathbb{R}^2}\int_{\mathbb{R}}q_{ij}(t,x',y_3)\,e^{-ix'\cdot\xi}e^{-it\cdot \tau}dt\,dx'.
\end{equation*}
With this, we have the following result:
\begin{lemma}\label{estimate_on_FourierT_of_q}
    There exist constants $C>0$, $\beta>0,\,\delta>0$ and $\rho_0>0$, such that the following estimate holds for any $\rho\geq\rho_0$ and fixed $y_3\in\R$,
    \begin{equation}\label{estimate on fourier}
       \Big|\widehat{q_{ij}}[(\cdot,\cdot,y_3)](\tau,\xi)\Big|\leq C \left(\rho^{\beta}\|\Lambda_{q^{(2)}} - \Lambda_{q^{(1)}}\| + \frac{1}{\rho^{\delta}}\right),\qquad (\tau,\xi)\in E,
   \end{equation}
   where the constant $C$ depends only on $\Omega,\,T$ and $M$.
\end{lemma}
\begin{proof}
Let us consider $(\tau,\xi)\in E$ and $\zeta\in\mathbb{S}^1$ be such that $\xi\cdot\zeta=0$. We define $\theta$ as $$\theta=\frac{\tau}{|\xi^2|}\cdot\xi+\sqrt{1-\frac{\tau^2}{|\xi^2|}}\cdot\zeta,$$
so that $\theta\in\mathbb{S}^{1}$ and $\theta\cdot\xi=\tau$.\\
    By the change of variable $x'=y'-t\theta$, we have for all $\xi\in\mathbb{R}^2,\, \theta\in\mathbb{S}^{1}$,
    \begin{equation*}
       \begin{split}
            \int_{\mathbb{R}^2} \mathcal{L}[q_{ij}(\cdot,\cdot,y_3)](\theta,y')\,e^{-iy'\cdot\xi}\,dy'&= \int_{\mathbb{R}^2}\bigg(\int_{\mathbb{R}}q_{ij}\,(t,\,y'-t\theta,\,y_3)\,dt \bigg)\,e^{-iy'\cdot\xi}\,dy'\\
            & = \int_{\mathbb{R}} \int_{\mathbb{R}^2}q_{ij}\,(t,\,x',\,y_3)\,e^{-ix'\cdot\xi}e^{-it\cdot \tau}dt\,dx'\\
            &=\widehat{q_{ij}}[(\cdot,\cdot,y_3)](\tau,\xi).
       \end{split}
    \end{equation*}
    Since $\omega$ is bounded in $\R^2$ and supp $q_{ij}(t,\cdot,y_3) \subset\bar{\omega}$. Hence there exists $\lambda>0$ such that $$\int_{\mathbb{R}^2\cap B(0,\lambda)} \mathcal{L}[q_{ij}(\cdot,\cdot,y_3)](\theta,y')\,e^{-iy'\cdot\xi}\,dy'=\widehat{q_{ij}}[(\cdot,\cdot,y_3)](\tau,\xi).$$
    Taking into account the estimate \eqref{estimate on ray R} in Lemma \ref{lem:esti_q}, we obtain the required estimate of  $\widehat{q_{ij}}(\cdot,\cdot,y_3)$.
\end{proof}
\subsection{Stability estimate}
In this section, we provide the proof of Theorem \ref{main result}. For $\kappa>0$, we denote $B(0,\kappa)=\{x\in\mathbb{R}^{3}: |x|<\kappa\}$. Our approach is similar to the one described in \cite{aicha2015stability} for the scalar potential.
The following lemma will play a crucial role in our further analysis.
\begin{lemma}[\cite{Hadmard 3 circle,vessella1999continuous}]\label{stability_for_analytic_continuation}
    Let $\mathcal{O}$ be a non empty open set of $B(0,1)$, and let $F$ be an analytic function in $B(0,2)$, obeying $$\|\partial^{\gamma}F\|_{L^{\infty}(B(0,2))}\leq \frac{M\,|\gamma|!}{\eta^{|\gamma|}},\quad\quad \forall \, \gamma\in(\mathbb{N}\cup \{0\})^{3}, $$
    for some $M>0$, and $\eta>0$. Then we have $$\|\partial^{\gamma}F\|_{L^{\infty}(B(0,1))}\leq (2M)^{1-\mu}\|\partial^{\gamma}F\|^{\mu}_{L^{\infty}(\mathcal{O})},$$ 
    where $\mu\in(0,1)$ depends on $n,\eta$ and $|\mathcal{O}|$.
\end{lemma}
\begin{proof}[Proof of Theorem \ref{main result}]
Let us fix $y_3\in\mathbb{R}$ and for fixed $\alpha >0$, we define  
$$F_{\alpha}[y_3](\tau,\xi):=\widehat{q_{ij}}[(\cdot,\cdot,y_3)](\alpha(\tau,\xi)),~~(\tau,\xi)\in\mathbb{R}^{3}.$$
Since we have assumed that $q_{ij}\equiv 0$ outside $\Omega_T$, hence $q_{ij}(\cdot,\cdot,y_3)\in L^{\infty}(\R\times \R^2)$ for any $y_3\in\R$, therefore using the Paley-Wiener's theorem, we have that $F_{\alpha}[y_3]$ is analytic and 
for $\gamma\in (\mathbb{N}\cup \{0\})^{3}$, we have that
\begin{equation}\label{estimate on F}
    \begin{split}
        |\partial^{\gamma}F_{\alpha}[y_3](\tau,\xi)|&=|\partial^{\gamma}\widehat{q_{ij}}[(\cdot,\cdot,y_3)](\alpha(\tau,\xi))|\\&=\Big|\partial^{\gamma}\int_{\mathbb{R}}\int_{\mathbb{R}^{2}}q_{ij}(t,x',y_3)e^{-i\alpha(t,x')\cdot(\tau,\xi)}\,dx'\,dt\Big|\\
        &=\Big|\int_{\mathbb{R}}\int_{\mathbb{R}^{2}}q_{ij}(t,x',y_3)(-i)^{|\gamma|}\alpha^{|\gamma|}(t,x')^{\gamma}e^{-i\alpha(t,x')\cdot(\tau,\xi)}\,dx'\,dt\Big|,
    \end{split}
\end{equation}
Now using the fact that $T>diam(\omega)$, we see that 
\begin{align*}
|\partial^{\gamma}F_{\alpha}[y_3](\tau,\xi)|& \leq \int_{\mathbb{R}}\int_{\mathbb{R}^{2}}|q_{ij}(t,x',y_3)|\alpha^{|\gamma|}(|x'|^2+t^2)^{\frac{|\gamma|}{2}}\,dx'\,dt\\
&\leq \|q_{ij}(\cdot,\cdot,y_3)\|_{L^1([0,T]\times\omega)}\,\alpha^{|\gamma|}\,(2\,T^2)^{\frac{|\gamma|}{2}}\\
&\leq C\|q_{ij}(\cdot,\cdot,y_3)\|_{L^\infty([0,T]\times\omega)}\,\alpha^{|\gamma|}\,(2\,T^2)^{\frac{|\gamma|}{2}}\\
       & \leq C \|q_{ij}\|_{L^{\infty}(\mathbb{R}_{y_3};L^\infty([0,T]\times\omega))}\,\, \frac{|\gamma|!}{(T^{-1})^{|\gamma|}}e^{\alpha}
       \leq C\,T^{|\gamma|}|\gamma|!e^{\alpha}.
\end{align*}
Setting $M=C\,e^{\alpha},\, \eta=T^{-1}$ and $\mathcal{O}=E^o\cap B(0,1)$, where $$E^o=\Big\{(\tau,\xi)\in\mathbb{R}\times (\mathbb{R}^2\backslash\{(0,0)\}),\,\, |\tau|\,<|\xi|\Big\},$$
and subsequently, applying the Lemma \ref{stability_for_analytic_continuation}, we find a constant $\mu\in(0,1)$ such that
\begin{align}\label{stab:temp_1}
|F_{\alpha}[y_3](\tau,\xi)|=|\widehat{q_{ij}}[(\cdot,\cdot,y_3)](\alpha(\tau,\xi))|\leq C\,e^{\alpha (1-\mu)}\|F_{\alpha}[y_3]\|^{\mu}_{L^{\infty}(\mathcal{O})},\quad (\tau,\xi)\in B(0,1).
\end{align}
Since $\alpha E^o=\{\alpha(\tau,\xi): (\tau,\xi)\in E^o\}=E^o$, using \eqref{stab:temp_1} we deduce for $(\tau,\xi)\in B(0,\alpha)$,
\begin{align*}
        |\widehat{q_{ij}}[(\cdot,\cdot,y_3)](\tau,\xi)|=\big|F_{\alpha}[y_3](\alpha^{-1}(\tau,\xi))\big|&\leq C\,e^{\alpha (1-\mu)}\|F_{\alpha}[y_3]\|^{\mu}_{L^{\infty}(\mathcal{O})}\\
        &\leq C\,e^{\alpha (1-\mu)}\|\widehat{q_{ij}}[(\cdot,\cdot,y_3)]\|^{\mu}_{L^{\infty}\big(B(0,\alpha)\cap E^o\big)} \\
        &\leq C\,e^{\alpha (1-\mu)}\|\widehat{q_{ij}}[(\cdot,\cdot,y_3)]\|^{\mu}_{L^{\infty}( E^o)}.
\end{align*}
This gives us the following estimate 
\begin{align}\label{est on ft}
    \|\widehat{q_{ij}}[(\cdot,\cdot,y_3)]\|_{L^{\infty}(B(0,\alpha))}\leq C\,e^{\alpha (1-\mu)}\|\widehat{q_{ij}}[(\cdot,\cdot,y_3)]\|^{\mu}_{L^{\infty}( E^o)}.
\end{align}
Next, we observe that 
\begin{align*}
    \begin{aligned}
        \|q_{ij}[(\cdot,\cdot,y_3)]\|^{2}_{H^{-1}(\mathbb{R}^{3})}&:=\int_{\R^{3}}(1+\lvert (\tau,\xi)\rvert^2)^{-1}\lvert \widehat{q_{ij}}[(\cdot,\cdot,y_3)](\tau,\xi)\rvert^2d\tau d\xi\\
        &= \int_{\lvert(\tau,\xi)\rvert< \alpha}(1+\lvert (\tau,\xi)\rvert^2)^{-1}\lvert \widehat{q_{ij}}[(\cdot,\cdot,y_3)](\tau,\xi)\rvert^2d\tau d\xi \\ &\ \ +\int_{\lvert(\tau,\xi)\rvert\geq  \alpha}(1+\lvert (\tau,\xi)\rvert^2)^{-1}\lvert \widehat{q_{ij}}[(\cdot,\cdot,y_3)](\tau,\xi)\rvert^2d\tau d\xi\\
        &\leq C\left(\alpha^3 \|\widehat{q_{ij}}[(\cdot,\cdot,y_3)]\|^2_{L^{\infty}(B(0,\alpha))}+\alpha^{-2}\|\widehat{q_{ij}}[(\cdot,\cdot,y_3)]\|^2_{L^{2}(\R^{3})}\right).
    \end{aligned}
\end{align*}
Now using the Plancherel theorem together with the boundedness of $q_{ij}$, we get 
\begin{align*}
    \begin{aligned}
     \|q_{ij}[(\cdot,\cdot,y_3)]\|^{2/\mu}_{H^{-1}(\mathbb{R}^{3})}  & \leq C \bigg( \alpha^{3}\|\widehat{q_{ij}}[(\cdot,\cdot,y_3)]\|^{2}_{L^{\infty}(B(0,\alpha))}+\alpha^{-2}\|{q}_{ij}[(\cdot,\cdot,y_3)]\|^{2}_{L^{\infty}([0,T]\times \omega)}\bigg)^{1/\mu}\\
        & \leq C \bigg( \alpha^{3}\|\widehat{q_{ij}}[(\cdot,\cdot,y_3)]\|^{2}_{L^{\infty}(B(0,\alpha))}+\alpha^{-2}\bigg)^{1/\mu}
    .\end{aligned}
\end{align*}
This after using equation \eqref{est on ft} and \eqref{estimate on fourier} gives us 
\begin{equation*}
   \begin{split}
        \|q_{ij}[(\cdot,\cdot,y_3)]\|^{2/\mu}_{H^{-1}(\mathbb{R}^{3})}&\leq C \bigg( \alpha^{3} \,e^{2\alpha (1-\mu)}\|\widehat{q_{ij}}[(\cdot,\cdot,y_3)]\|^{2\mu}_{L^{\infty}( E^o)} +\alpha^{-2} \bigg)^{1/\mu}\\
        &\leq C \bigg( \alpha^{3} \,e^{2\alpha (1-\mu)}\left(\rho^{\beta}\|\Lambda_{q^{(2)}} - \Lambda_{q^{(1)}}\| + \frac{1}{\rho^{\delta}}\right)^{2\mu}+\alpha^{-2} \bigg)^{1/\mu}\\ &\leq C \bigg( \alpha^{\frac{3}{\mu}} \,e^{\frac{2\alpha (1-\mu)}{\mu}}\rho^{2\beta}\|\Lambda_{q^{(2)}} - \Lambda_{q^{(1)}}\|^2 + \alpha^{\frac{3}{\mu}} \,e^{\frac{2\alpha (1-\mu)}{\mu}} \frac{1}{\rho^{2\delta}}+\alpha^{-2/\mu} \bigg),
   \end{split}
\end{equation*}
where $y_3\in \mathbb{R}$.
Let $\alpha_0 >0$ be sufficiently large and consider $\alpha>\alpha_0$. We set $\alpha$ in such a way that
$$\alpha^{\frac{3}{\mu}} \,e^{\frac{2\alpha (1-\mu)}{\mu}} \frac{1}{\rho^{2\delta}}=\alpha^{-2/\mu}.$$
As a consequence, we have
$$\rho=\alpha^{\frac{5}{2\mu\delta}} \,e^{\frac{\alpha (1-\mu)}{\mu\delta}}$$
and 
\begin{equation}\label{H-1 estimate}
    \begin{split}
        \|q_{ij}[(\cdot,\cdot,y_3)]\|^{2/\mu}_{H^{-1}(\mathbb{R}^{3})} &\leq C \bigg( \alpha^{\frac{3\delta+5\beta}{\delta\mu}} \,e^{\frac{2\alpha(\delta+\beta)(1-\mu)}{\delta\mu}}\|\Lambda_{q^{(2)}} - \Lambda_{q^{(1)}}\|^2+\alpha^{-2/\mu} \bigg),\qquad y_3\in \mathbb{R},\nonumber\\
        & \leq C \big( e^{N\alpha}\|\Lambda_{q^{(2)}} - \Lambda_{q^{(1)}}\|^2 +\alpha^{-2/\mu}\big),\qquad y_3\in \mathbb{R},
    \end{split}
\end{equation}
where $N$ depends on $\delta,\beta$ and $\mu$. Now for a fixed $\alpha>0,$ large enough, we choose $0<c<1$ such that $$0<\|\Lambda_{q^{(2)}} - \Lambda_{q^{(1)}}\|<c\,$$ and  
\[e^{N\alpha}\|\Lambda_{q^{(2)}} - \Lambda_{q^{(1)}}\|^2 +\alpha^{-2/\mu}\leq \|\Lambda_{q^{(2)}} - \Lambda_{q^{(1)}}\|^{\mu/2}+\Big|\log\|\Lambda_{q^{(2)}} - \Lambda_{q^{(1)}}\|\Big|^{-1}.\]
Using this in \eqref{H-1 estimate}, we get that 
\begin{align}\label{estima on q}
     \|q_{ij}[(\cdot,\cdot,y_3)]\|^{2/\mu}_{H^{-1}([0,T]\times\omega)}&=\|q_{ij}[(\cdot,\cdot,y_3)]\|^{2/\mu}_{H^{-1}(\mathbb{R}^{3})} \nonumber \\
     &\leq C \bigg(\|\Lambda_{q^{(2)}} - \Lambda_{q^{(1)}}\|+\Big|\log\|\Lambda_{q^{(2)}} - \Lambda_{q^{(1)}}\|\Big|^{-2/\mu}\bigg)^{\mu/2},\nonumber\\
     & \leq C \bigg(\|\Lambda_{q^{(2)}} - \Lambda_{q^{(1)}}\|^{\mu/2}+\Big|\log\|\Lambda_{q^{(2)}} - \Lambda_{q^{(1)}}\|\Big|^{-1}\,\bigg),
\end{align}
where $y_3\in\mathbb{R}$\,. Since the right hand side of \eqref{estima on q} is independent of $y_3\in\mathbb{R}$, therefore we obtain
\begin{align*}
   \|q_{ij}[(\cdot,\cdot,y_3)]\|^{2/\mu}_{L^{\infty}\big(\mathbb{R}_{y_3};H^{-1}([0,T]\times\omega)\big)} \leq C \bigg(\|\Lambda_{q^{(2)}} - \Lambda_{q^{(1)}}\|^{\mu/2}+\Big|\log\|\Lambda_{q^{(2)}} - \Lambda_{q^{(1)}}\|\Big|^{-1}\,\bigg). 
\end{align*}

Also, if $\|\Lambda_{q^{(2)}} - \Lambda_{q^{(1)}}\|>c$, then we have 
\begin{align*}
    \|q_{ij}[(\cdot,\cdot,y_3)]\|_{H^{-1}([0,T]\times\omega)}&\leq C \|q_{ij}[(\cdot,\cdot,y_3)]\|_{L^{\infty}([0,T]\times\omega)}\\&\leq \frac{2CMc^{\mu/2}}{c^{\mu/2}}\leq \frac{2CM}{c^{\mu/2}}\|\Lambda_{q^{(2)}} - \Lambda_{q^{(1)}}\|^{\mu/2}\,,
\end{align*} hence \eqref{estima on q} holds. Thus combining the above estimates, we get 
\[\|q_{ij}[(\cdot,\cdot,y_3)]\|^{2/\mu}_{H^{-1}([0,T]\times\omega)}
\leq C \bigg(\|\Lambda_{q^{(2)}} - \Lambda_{q^{(1)}}\|^{\mu/2}+\Big|\log\|\Lambda_{q^{(2)}} - \Lambda_{q^{(1)}}\|\Big|^{-1}\,\bigg)\]
Now for $s>n/2$, using the Sobolev embedding theorem, we get that 
 \begin{equation}  \label{Sobolev embedding}      
 \|q_{ij}[(\cdot,\cdot,y_3)]\|_{L^{\infty}([0,T]\times\omega)} \leq C \|q_{ij}[(\cdot,\cdot,y_3)]\|_{H^{s}([0,T]\times\omega)}.
 \end{equation}
 Finally using the Sobolev interpolation theorem, we have that 
      \begin{align*}
          \begin{aligned}
              \|q_{ij}[(\cdot,\cdot,y_3)]\|_{H^{s}([0,T]\times\omega)}  & \leq C  \|q_{ij}[(\cdot,\cdot,y_3)]\|^{1-\beta}_{H^{-1}([0,T]\times\omega)}\|q_{ij}[(\cdot,\cdot,y_3)]\|^{\beta}_{H^{s+1}([0,T]\times\omega)}\\
          & \leq C  \|q_{ij}[(\cdot,\cdot,y_3)]\|^{1-\beta}_{H^{-1}([0,T]\times\omega)},
          \end{aligned}
      \end{align*} 
 where the parameter $\beta\in(0,1)$ and constant $C>0$ independent of $q_{ij}$. 
 Combining this with 
 \eqref{Sobolev embedding} and using  \eqref{estima on q}, we obtain
 \begin{align*}
     \|q_{ij}\|_{L^{\infty}(\Omega_T)} \leq C \bigg(\|\Lambda_{q^{(2)}} - \Lambda_{q^{(1)}}\|^{\mu/2}+\Big|\log\|\Lambda_{q^{(2)}} - \Lambda_{q^{(1)}}\|\Big|^{-1}\,\bigg).
 \end{align*}
Since the above estimate holds for every $i,j\in \{1,2,\dots,n\}$, we have demonstrated the required stability estimate. Hence the Theorem \ref{main result} follows.
 \end{proof} 
\subsubsection{Proof of Theorem \ref{main result_with_partial_data}}
Hereby, we give a proof of the Theorem \ref{main result_with_partial_data}. As GO solutions are one of the main ingredients for stability estimate, we require appropriate assumptions on GO solutions $\overrightarrow {u}_{\rho}$ such that $\overrightarrow {u}_{\rho}|_{\Sigma_R}\in \LS_R$. More precisely, we will modify $h\in\mathcal{S}(\mathbb{R})$ in the Lemma \ref{lemma 1 for stab} for our analysis.
\begin{lemma}\label{lemma 1 for stab_with_partial_data}
Let $q^{(1)},q^{(2)}\in W^{1,\infty}(\Omega_T)$, with $\|q^{(j)}\|\leq M, j=1,2,$ and let matrix potential $q$ of size $n\times n$ be equal to $q^{(1)}-q^{(2)}$ and extended by zero outside of $\Omega_T$ . Then, for all $\overrightarrow {K}, \overrightarrow {K}^{(*)}\in\mathbb{R}^n$, $\theta \in  \mathbb{S}^1$, $h \in \mathcal{C}^{\infty}_0((-R,R))$  and $\varphi\in \mathcal{C}_{0}^{\infty}(\mathbb{R}^{2})$ we have 
\begin{equation}\label{first ineq for stab_with_partia_data}
  \begin{split}
\biggr|\int_{\mathbb{R}}\int_{\mathbb{R}^2}\int_{0}^{T}\Big(q(t,x',x_3) &\overrightarrow {K}\cdot\overrightarrow {K}^{(*)}\Big)\varphi^{2}(x'+ t\theta)h^2(x_3)\,dt\,dx'\,dx_3\biggr|\\ 
& \leq C\bigg(\frac{1}{\rho} + \rho^4 \|\Lambda^{(R)}_{q^{(2)}} - \Lambda^{(R)}_{q^{(1)}}\| \bigg)\|\varphi\|^2_{H^{3}(\mathbb{R}^{2})}\|h\|^2_{H^{2}(\mathbb{R})},
  \end{split}
\end{equation}
for any $\rho >1$ sufficiently large.  Here the constant $C>0$ depends on $R,\Omega, T, M, \overrightarrow {K}$ and $\overrightarrow {K}^{(*)}$.
\end{lemma}
With the help of Lemma \ref{lemma 1 for stab_with_partial_data}, we complete the proof of Theorem \ref{main result_with_partial_data}.
\begin{proof}[Proof of Theorem \ref{main result_with_partial_data}]
 Let us consider the matrix potential $q$ as introduced in \eqref{extending_q_outside_of_the_domain} with taking values zero outside of $\Omega_T$. In addition, the functions $\varphi_\eps$ and $h_\eps$, which were introduced in the Lemma \ref{lem:esti_q}, are taken into consideration with $y_3\in(-r,r)$. Therefore, combining Lemma \ref{lemma 1 for stab_with_partial_data} with Lemma \ref{lem:esti_q} and Lemma \ref{estimate_on_FourierT_of_q}, we obtain the estimate
  \begin{equation}\label{estimate on fourier_with_partial_boundary_data}
       \Big|\widehat{q_{ij}}[(\cdot,\cdot,y_3)](\tau,\xi)\Big|\leq C \left(\rho^{\beta}\|\Lambda^{(R)}_{q^{(2)}} - \Lambda^{(R)}_{q^{(1)}}\| + \frac{1}{\rho^{\delta}}\right),\qquad (\tau,\xi)\in E,
   \end{equation}
   where $\beta,\delta >0$ and  $\rho\geq\rho_0>0$. Afterwards, following the proof of Theorem \ref{main result}, we obtain for $i,j\in\{1,2,\dots,n\},$  
   \begin{align}\
     \|q_{ij}\|^{2/\mu}_{L^{\infty}((0,T)\times\omega\times(-r,r))}\leq C \bigg(\|\Lambda^{(R)}_{q^{(1)}} - \Lambda^{(R)}_{q^{(2)}}\|^{\mu/2}+\Big|\log\|\Lambda^{(R)}_{q^{(1)}} - \Lambda^{(R)}_{q^{(2)}}\|\Big|^{-1}\,\bigg),
\end{align}
where $C>0$ depends on $R,\Omega, M$ and $T$. Taking into account the assumption \eqref{q_in_bounded_domain}, we have the estimate \eqref{main_stability_estimate_for_q_with_partial_boundary_data}.
\end{proof}
\section*{Acknowledgements}
MV is supported by ISIRD
project from IIT Ropar and Start-up Research Grant SRG/2021/001432 from the Science and Engineering
Research Board, Government of India. The authors would like to thank the anonymous reviewers for their valuable remarks and suggestions.

\end{document}